\documentclass[amsart]{article}

\usepackage[english]{babel}
\usepackage[utf8x]{inputenc}
\usepackage{lmodern}
\usepackage[T1]{fontenc}
\usepackage{lmodern} 

\usepackage[a4paper,top=3cm,bottom=2cm,left=2.3cm,right=2.3cm,marginparwidth=1.75cm]{geometry}

\usepackage{amsmath}
\usepackage{graphicx}
\usepackage[hidelinks]{hyperref} 
\usepackage{amscd}     
\usepackage{amsfonts}
\usepackage{dsfont}
\usepackage{latexsym}
\usepackage{amscd}
\usepackage{verbatim}
\usepackage{graphicx}
\usepackage{appendix}

\usepackage{amsmath}
	\numberwithin{equation}{section}
\usepackage{amssymb}
\usepackage{amsthm}
	\theoremstyle{plain}
		\newtheorem{thm}{Theorem}[section]
		\newtheorem{coro}[thm]{Corollary}
		\newtheorem{lem}[thm]{Lemma}
		\newtheorem{prop}[thm]{Proposition}
        
	\theoremstyle{definition}
		\newtheorem{defn}[thm]{Definition}
		\newtheorem{ex}[thm]{Example}
		\newtheorem{notation}[thm]{Notation}
        
	\theoremstyle{remark}
		\newtheorem{rmk}[thm]{Remark}
		\newtheorem{remark}[thm]{Remark}

\usepackage{mathtools} 
\usepackage{mathrsfs} 
\usepackage{eucal} 
\usepackage{microtype}

\usepackage{mparhack}
	
\usepackage{xspace}

\usepackage{caption}
\usepackage{varioref}		
\usepackage{enumitem}

\usepackage{tikz} 
\usepackage{tikz-cd}

\newcommand{\N}{\mathbb{N}}
\newcommand{\Z}{\mathbb{Z}}

\newcommand{\F}{\mathbb{F}}
\newcommand{\K}{\mathbb{K}}
\newcommand{\C}{\mathcal{C}}
\newcommand{\G}{\mathcal{G}}

\newcommand{\V}{\mathcal{V}}
\renewcommand{\=}{\coloneqq}

\DeclareMathOperator{\im}{im}	
\DeclareMathOperator{\id}{id}
	
\DeclareMathOperator{\coker}{coker}

\DeclareMathOperator{\supp}{supp}	

\DeclareMathOperator{\cone}{Cone}
\DeclareMathOperator{\Tor}{Tor}

\title{On the support of Betti tables of multiparameter persistent homology modules}

\author{Andrea Guidolin\footnote{University of Southampton, United Kingdom}\\ \texttt{a.guidolin@soton.ac.uk}
\and Claudia Landi\footnote{University of Modena and Reggio Emilia, Italy}\\
\texttt{claudia.landi@unimore.it}}

\date{}

\begin{document}
\maketitle

\begin{abstract}
Persistent homology encodes the evolution of homological features of a multifiltered cell complex in the form of a multigraded module over a polynomial ring, called a multiparameter persistence module, and quantifies it through invariants suitable for topological data analysis. 
  
In this paper, we establish relations between the Betti tables, a standard invariant for multigraded modules commonly used in multiparameter persistence, and the multifiltered cell complex. In particular, we show that the grades at which cells of specific dimensions first appear in the filtration reveal all positions in which the Betti tables are possibly nonzero. This result can be used in combination with discrete Morse theory on the multifiltered cell complex originating the module to obtain a better approximation of the support of the Betti tables. In the case of bifiltrations, we refine our results by considering homological critical grades of a filtered chain complex instead of entrance grades of cells.\\
  
 
\end{abstract}

\section{Introduction}
\label{sec:intro}

One of the main concepts in Topological Data Analysis is {\em persistent homology}, a tool to capture topological information at multiple scales and provide meaningful topological summaries of the data, as surveyed, for example, in \cite{Ghrist2008, Carlsson2009TD, edelsbrunner2012persistent}. 
In practice, assuming that a data set comes equipped with measurements like functions or metrics to filter it, persistent homology transforms the filtered data into a nested family of chain complexes that depend on as many parameters as the number of different measurements used. 
Applying homology with coefficients in a field $\F$ to such a filtered chain complex produces a parametrized family of vector spaces, connected by linear transition maps, called a {\em persistent homology module}. 
Algebraic invariants of persistent homology modules provide the required summaries of the data topology. 

Classically, the development of the theory of persistent homology originated from two separate roots: Morse theory (as in, e.g., \cite{Barannikov94,frosini1996connections,Robins2000, Edelsbrunner2002}), and commutative algebra (as in, e.g., \cite{zomorodian2005computing,carlsson2009theory,oudot2015persistence}). These two perspectives reconcile very elegantly in the case of 1-parameter persistence, i.e.\ when the filtration depends on only one parameter. In this case,  persistent homology modules admit a complete invariant,  the  so-called {\em barcode},  encoding the lifespan of homology classes through the considered filtration.
From the standpoint of Morse theory, the endpoints of bars in a persistence barcode correspond to the cancellation of pairs of critical points of the filtering (Morse) function \cite{Edelsbrunner2002}.  From the algebraic  perspective,  a  persistent homology module is a representation of a finite linear quiver in the category of vector spaces. Thus, a 1-parameter persistence module  admits a unique decomposition  into  {\em interval modules}, i.e.\ indecomposable modules, each  supported on an interval. These intervals are exactly the bars of the persistence barcode \cite{zomorodian2005computing}. 

It is of both theoretical and practical interest to understand persistent homology in the case of multiple parameters, yielding to the so-called {\em multiparameter persistence}. Indeed,  in applications, one often needs to filter the data using more than only one measurement, obtaining a multiparameter persistence module. This is the case, for example, when there are different drivers for a phenomenon \cite{Cerri2013PHOG}, or when one needs to downsize the role of outliers by adding a co-density measurement to the principal, explanatory, measurement as in \cite{Frosini2013,Blumberg2022}. 

Unfortunately, the theory of multiparameter persistence modules proves to be much more elusive than the single-parameter one: in particular, since multigraded modules are of wild representation type \cite{gabriel}, more complicated indecomposables than just intervals can generally occur, and it is impossible to list them all or characterize them via discrete invariants.
Despite this difficulty, all the relevant  
homological events in a multiparameter filtration are conveniently captured by the {\em Betti tables} of the multiparameter persistent homology module  \cite{carlsson2009theory}. However, these events cannot be paired to obtain summaries similar to barcodes, and their mutual dependencies cannot be easily unveiled.

One of the motivations of this paper is to relate the events captured by the Betti tables of a multiparameter persistent homology module to the events captured by Morse theory, considered in its combinatorial formulation \cite{Forman1998, kozlov2005discrete}. This attempt to reconnect the algebraic perspective to Morse theory in the multiparameter situation is both of theoretical interest in commutative algebra and of practical advantage, as it provides a unified perspective to study persistent homology modules together with the underlying filtered complexes.

In this perspective, starting from the observation that for a 1-parameter persistent homology module the support of the Betti tables  coincides with the set of
entrance grades of critical cells in the filtration under consideration,  our goal is to understand whether and to what extent this fact can be generalized to multiparameter persistence. 
An indication that this may be the case comes from the  results of \cite{guidolin2023morse}, which establish Morse inequalities involving, on the one hand,  the values of the Betti tables of a multiparameter persistent homology module, and, on the other hand,  the so-called  \emph{homological critical numbers} of the same filtration. The latter numbers can be viewed as theoretical lower bounds of the numbers of critical cells entering the filtration at each filtration grade for any choice of a discrete gradient vector field consistent with the filtration.

The results of this paper delimit, in the space of parameters, the support of the Betti tables of a persistent homology module in terms of the entrance grades of cells in the multiparameter filtration. Moreover, we study the relation between the dimension of the entered cells and the degree of the persistent homology module on which they impact. In our setting, the multiparameter filtration is defined on an abstract cell complex, an object representing in a combinatorial way a chain complex of vector spaces with distinguished bases (Section \ref{subsec:chaincm}). To obtain our main results of Sections \ref{sec:free}, \ref{sec:support_xi} and \ref{sec:2param}, the filtration is assumed to be defined via the sublevel sets of measurement functions.
In such filtrations, also called one-critical in topological data analysis \cite{carlsson2009computing}, every cell has a unique entrance grade (Section \ref{subsec:multipers}). 

From a different perspective, we aim to highlight how prior known results about multigraded resolutions are relevant to the study of multiparamenter persistence, and what can be gained in the context of persistence by integrating them with Morse theory. Indeed, the main goal of this paper can be stated also in the language of multigraded commutative algebra, considering $n$-graded modules over the polynomial ring $S\coloneqq \F[x_1,\ldots ,x_n]$. An $n$-parameter persistent homology module can be viewed as an $n$-graded $S$-module $V$ which is presented as the homology at the middle term of a sequence $A\xrightarrow{f}B\xrightarrow{g}C$ of $n$-graded $S$-modules with $gf=0$. If the $n$-parameter filtration is one-critical, the modules $A$, $B$, and $C$ are free. Our goal is to study the Betti tables of $V$ and relate their support with the grades of the generators of the modules $A$, $B$, and $C$.

In Section \ref{sec:free}, we highlight how multigraded free presentations and resolutions, well-studied in multigraded commutative algebra \cite{miller2005combinatorial,peeva2010graded}, can be applied in the context of multiparameter persistence. Via this approach, we obtain some initial bounds on the support of the Betti tables of a persistent homology module in terms of entrance grades of cells in the multiparameter filtration (Proposition \ref{prop:suppfreepresVq} and Remark \ref{rmk:free-bounds}).

Nevertheless, we can say more about the support of Betti tables of persistence if, instead of approaching the problem directly using a free resolution of the multiparameter persistence module, we use the \emph{Koszul complex} associated with the persistence module, a strategy already used in \cite{knudson2008refinement,lesnick2019computing,guidolin2023morse}.  More specifically, our technique is based on the construction of the Koszul complex via mapping cones (Section \ref{sec:Koszul}). Using this inductive construction, we can compute Betti tables by looking at the space of parameters only locally and, more importantly, we can disentangle the different parameters of the multiparameter filtration: the Koszul complex at a fixed grade in an $n$-parameter space is determined by the Koszul complexes at nearby grades in an $(n-1)$-parameter space. This allows for explicit and direct proofs. As an advantage, we can identify obstructions to the vanishing of Betti tables of a persistence module, which may not be as clear using the more abstract approach via free resolutions, and get tighter bounds than directly using resolutions. 

In detail, given an $n$-parameter filtration $\{ X^u\}_{u\in \N^n}$ of a finite cell complex $X$, we consider, for any $q\in \N$, the set $\G(X_q)$ of entrance grades of $q$-cells in the filtration, as well as its closure $\overline{\G (X_q)}$ with respect to least upper bounds, i.e.\ the smallest set containing $\G(X_q)$ and the least upper bounds in $\N^n$ of its nonempty subsets. We denote by $\xi_i^q\colon\N^n\to \N$ the $i$th Betti table of the persistent homology module obtained as the $q$th homology of the filtration.
In the case when the filtration is one-critical, Theorem \ref{thm:supp_xi_m} of Section \ref{sec:support_xi} states a relation between the support $\supp \xi_i^q \coloneqq \{u\in \N^n \mid \xi_i^q(u)\ne 0 \}$ of the Betti tables and the sets of entrance grades of cells: for all $q\in \N$,
\[ \bigcup_{i=0}^n \supp \xi^q_i   \subseteq \overline{\G (X_{q+1})} \cup \overline{\G (X_{q})}.
\]
This delimitation of the support of the Betti tables using the entrance grades of cells cannot be tightened (see Example \ref{ex:tight}).  

We next focus on particular Betti tables for which the containment above can be improved. 
Still in Theorem \ref{thm:supp_xi_m}, we prove that $\supp \xi^q_0   \subseteq \overline{\G (X_q)}$ and $\supp \xi^q_n   \subseteq \overline{\G (X_{q+1})}$, for all $q\in \N$. More interestingly, in Theorem \ref{thm:H1K} we identify a sufficient condition on submodules of boundaries and cycles for the vanishing of $\xi_1^q$ at a grade $u\in \N^n$. The condition for boundaries is the identity $B_q(X^u) = \sum_{j=1}^n B_q(X^{u-e_j})$ of submodules of $C_q(X^u)$, while the condition for  cycles consists, up to a permutation on the set $\{1,\ldots ,n\}$ enumerating the parameters, of the identities
\[
Z_q(X^{u-e_{\ell}})\cap \left( \sum_{j< \ell}Z_q(X^{u-e_{j}})\right) = \sum_{j< \ell}Z_q(X^{u-e_{j}-e_{\ell}}) ,
\]
for every $\ell \le n$. Our result implies the bound $\supp \xi_1^q \subseteq \G(X_{q+1})\cup \overline{\G(X_q)}$ for the support of the $1$st Betti table (Corollary \ref{coro:H1K}).
In particular, in comparison to what can be obtained using multigraded free resolutions of the persistent homology module as in Section \ref{sec:free}, we see that using the cone construction of the Koszul complex we get somehow stronger results. 

To reconnect our results with Morse theory, in Section \ref{subsec:Morse-supp} we observe that all our bounds for the support of Betti tables can be applied to the Morse complex associated with any discrete gradient vector field consistent with the filtration. The persistent homology module of the Morse complex has the same Betti tables as that of the original filtration, but the set of entrance grades of cells is typically much smaller. Therefore, using Morse complexes, one can often obtain better approximations of the support of the Betti tables.

In the endeavor to improve the bounds for the support of Betti tables, rather than considering entrance grades of cells (either of the original complex or of an associated Morse complex), as a further contribution of this paper we show that, in the case of 2-parameter filtrations that are one-critical, the support of the Betti tables of a persistent homology module is contained in the closure of the set of \emph{homological critical grades} (Section~\ref{sec:2param}). Although limited to the case of two parameters, this result improves our results from Section \ref{sec:support_xi} in two ways: it does not depend on the choice of a specific discrete gradient vector field and  establishes that all events witnessed by the Betti tables are determined by homological criticality (Corollary \ref{coro:supp_2par}). 

Our results of Section \ref{sec:support_xi} and Section \ref{sec:2param} hold for one-critical filtrations of cell complexes. Although they cannot be applied directly to filtrations that are not one-critical, a generalization in this direction can be obtained using results from \cite{Chacholski2017}, as we explain in Section \ref{sect:not-one-crit}.

\section{Preliminaries}
\label{sec:prelim}

Before presenting relevant background material for this article, let us establish some general notations: $\N$ denotes the set $\{0,1,\ldots\}$ of natural numbers;
$[n]$ denotes the set $\{1, 2, \ldots , n\}$; $\{e_i\}_{i=1,\ldots ,n}$ is the standard basis of $\N^n$; for any subset $\alpha \subseteq [n]$, we denote $e_{\alpha}\= \sum_{j\in \alpha} e_j$; $|J|$ denotes the cardinality of a set $J$; the symbols $\wedge$ and $\vee$  denote the greatest lower bound and least upper bound, respectively.

\subsection{Based chain complexes, cell complexes, and homology}
\label{subsec:chaincm}

Let $\F$ denote a field, arbitrary but fixed. 
A \emph{based chain complex} is a chain complex $C_* = (C_q,\partial_q)_{q\in \Z}$ of vector spaces over $\F$, which we assume to be of finite dimension, such that each $C_q$ is endowed with a distinguished basis $X_q$. Throughout this article, we assume all chain complexes to be bounded, meaning that $C_q =0$ whenever $q<0$ or $q\ge m$ for some integer $m$. 
Based chain complexes can be viewed from a combinatorial perspective, as their distinguished bases inherit the structure of an (abstract) \emph{cell complex}, in the sense of Lefschetz \cite{lefschetz1942algebraic}.
In this work, we call cell complex a finite graded set $X=\bigsqcup_{q\in \N} X_q$, whose elements are called \emph{cells}, endowed with an \emph{incidence function} $\kappa : X\times X \to \F$. A cell $\sigma \in X_q$ is said to have  \emph{dimension} $q$, denoted $\dim \sigma = q$, or to be a $q$-\emph{cell}. The incidence function must satisfy two axioms: 
(\textit{i}) $\kappa (\tau , \sigma) \ne 0$ implies $\dim \tau = \dim \sigma + 1$, and 
(\textit{ii}) $\sum_{\rho\in X}\kappa (\tau , \rho)\cdot \kappa (\rho ,\sigma)=0$, for any pair of cells $\tau$ and $\sigma$ in $X$. 
We endow $X$ with the order relation $\le$, called the \emph{face partial order}, generated by the \emph{covering face relation}: $\sigma < \tau$ if and only if $\kappa (\tau , \sigma) \ne 0$.
Given a cell complex $X$, we denote $C_* (X) = (C_q(X),\partial_q)_{q\in \Z}$ the based chain complex such that $X_q$ is the fixed basis of $C_q$, for all $q$, with differentials $\partial_q : C_q \to C_{q-1}$ defined on each $\tau \in X_q$ by
\[ \partial_q (\tau) = \sum_{\sigma\in X_{q-1}} \kappa (\tau ,\sigma) \sigma .
\]
We observe that $C_*(X)$ is the zero chain complex if $X=\emptyset$.

A graded set $A=\bigsqcup_{q\in \N} A_q$ is called a \emph{subcomplex} of $X$ if, for all $\tau \in A$, every cell $\sigma \in X$ such that $\sigma \le \tau$ is also in $A$. This property makes $A$, endowed with the restriction of the incidence function of $X$, a cell complex, and is equivalent to requiring $C_* (A)$ to be a chain subcomplex of $C_* (X)$.
We denote by $H_q(X) := \ker \partial_q / \im \partial_{q+1}$ the homology $\F$-modules of $C_*(X)$, and by $H_q(X,A)$ the homology $\F$-modules of the relative chain complex $C_*(X,A)$.

We observe that the notion of a cell complex as reviewed above, equivalent to that of a based chain complex, is general enough to include simplicial complexes and cubical complexes, among other widely used combinatorial objects admitting a canonically associated chain complex. If the aim is computing homology, finite CW complexes can also be represented by cell complexes, letting $\kappa (\tau,\sigma)$ be the degree of the attaching map from the boundary of $\tau$ to $\sigma$.

\subsection{Multifiltrations and multiparameter persistence}
\label{subsec:multipers}

One of the main mathematical objects of interest in topological data analysis are functors from a poset to the category of finite dimensional vector spaces over a field $\F$. 
Here, we consider the indexing poset $\N^n$, for some integer $n\ge 1$, equipped with the coordinate-wise partial order: for $u=(u_i),v=(v_i)\in \N^n$, we write $u\preceq v$ if and only if $u_i \le v_i$, for all $1\le i\le n$. 
In this article, an $n$-\emph{parameter persistence module} is a functor from the poset $(\N^n,\preceq)$ with values in finite-dimensional $\F$-vector spaces. Morphisms between such functors are the natural transformations. 
Explicitly, an $n$-parameter persistence module $V$ consists of a family $\{ V^u \}_{u\in \N^n}$ of $\F$-vector spaces together with a family $\{ \varphi^{u,v} : V^u \to V^v \}_{u\preceq v \in \N^n}$ of linear maps such that $\varphi^{u,w} = \varphi^{v,w} \circ \varphi^{u,v}$ whenever $u\preceq v\preceq w$, and $\varphi^{u,u} = \id_{V^u}$, for all $u$. 
A {\em morphism} between two $n$-parameter persistence modules $\{ V^u ,\varphi^{u,v} \}$ and $\{ W^u ,\psi
^{u,v} \}$ is a family of linear maps $\{ \nu^u : V^u \to W^u \}_{u\in \N^n}$ such that $\nu^v\circ \varphi^{u,v}=\psi^{u,v}\circ \nu^u$, for all $u\preceq v$ in $\N^n$. A morphism $\nu$ is an {\em isomorphism} ({\em monomorphism}, {\em epimorphism}, respectively) if, and only if, its components $\nu^u$ are bijective (injective, surjective), for all $u\in \N^n$. 

In topological data analysis, the typical source of persistence modules are filtrations of cell complexes associated with the data. 
An $n$-\emph{filtration} of a cell complex $X$ is a family $\{ X^u \}_{u \in \N^n}$ of subcomplexes of $X$ such that $u\preceq v$ implies $X^u \subseteq X^v$. 
If a cell $\sigma$ of $X$ is an element of $X^u \smallsetminus \bigcup_{j=1}^n X^{u-e_j}$, we say that $u$ is an {\em entrance grade} of $\sigma$ in the filtration. In this article we assume, unless otherwise stated, that filtrations $\{ X^u \}_{u \in \N^n}$ are families of sublevel sets $X^u = \{ \sigma \in X \mid h(\sigma) \preceq u \}$ of some order-preserving function $h:(X,\le) \to (\N^n ,\preceq)$, with $\le$ denoting the face partial order on $X$. This assumption is equivalent to requiring every cell of $X$ to have exactly one entrance grade, and will only be lifted in Section \ref{sect:not-one-crit}, where we discuss applications to general $n$-filtrations.

The filtrations we are considering are usually called \emph{one-critical} \cite{carlsson2009computing} in topological data analysis. We want to highlight that assuming the uniqueness of entrance grades is fundamental in order to obtain the results of Section \ref{sec:support_xi} and Section \ref{sec:2param}, which are false for general filtrations of cell complexes (but can be adapted as explained in Section \ref{sect:not-one-crit}). 
For instance, in this article we repeatedly use the following fact. 
\begin{rmk}
\label{rmk:one-criticality}
Given a one-critical $n$-filtration $\{ X^u \}_{u\in \N^n}$ and a finite set of filtration grades $\{ u_j  \}_{j=1,\ldots,k} \subseteq \N^n$, with  $u_j = (u_{j,1}, \ldots ,u_{j,n})$ for all $j$, we have
$\bigcap_{j=1}^k X^{u_j} = X^{w}$,
where $w = \bigwedge \{u_j \}_j = (\min \{u_{j,1}\}_j, \ldots ,\min \{u_{j,n}\}_j)$ is the greatest lower bound of the subset $\{ u_j  \}_{j=1,\ldots,k}$ in $\N^n$. 
In particular, for each subset $\alpha \subseteq [n]$, we have the equality 
$\bigcap_{j\in \alpha} X^{u-e_j} = X^{u-e_\alpha}$. 
\end{rmk}

We are interested in persistence modules obtained as the homology of an $n$-filtration.
Given an $n$-filtration $\{ X^u \}_{u \in \N^n}$ and applying the $q$th homology functor, one obtains the $n$-\emph{parameter persistent $q$th-homology module} ${V}_q=\{ {V}_q^u, \iota^{u,v}_{q}\}_{u\preceq v\in \N^n}$, with ${V}_q^u :=H_q (X^u)$ and 
$\iota^{u,v}_{q} \colon H_q(X^u) \to H_q (X^v)$ induced by the inclusion maps $X^u \hookrightarrow X^v$ for $u\preceq v$. We note that it is common to use the terms \emph{multifiltration} and \emph{multiparameter} in place of, respectively, $n$-filtration and $n$-parameter, to indicate the generic case when $n>1$. Moreover, $2$-filtrations are also called \emph{bifiltrations}.  

The overall purpose of this work is to study the relation between the homological invariants of multiparameter persistent homology modules called Betti tables and the multifiltrations from which they are obtained. To this aim, we adopt some tools and terminology from commutative algebra. An $n$-\emph{graded module} over the polynomial ring $S:=\F [x_1, \ldots ,x_n]$ is an $S$-module with a vector space decomposition $V = \bigoplus_{u\in \N^n}V^u$ such that $x_i\cdot V^u \subseteq V^{u+e_i}$, for all $u\in \N^n$ and $i\in [n]$. There is a standard equivalence \cite{carlsson2009theory} between the category of $n$-parameter persistence modules and the category of $n$-graded $S$-modules, allowing us to view a persistence module $\{ V^u ,\varphi^{u,v} \}$ as the $n$-graded $S$-module $\bigoplus_{u\in \N^n} V^u$, where the action of $S$ is defined by $x_i \cdot z = \varphi^{u,u+e_i} (z)$, for all $z\in V^u$ and $i\in [n]$.
Standard homological invariants from commutative algebra, like the \emph{Betti tables} (also called \emph{multigraded Betti numbers}, see Section \ref{subsec:resolutions}), were among the first ones studied in multiparameter persistence \cite{carlsson2009theory,knudson2008refinement}. 
Given an $n$-parameter persistent homology module $\{ V_q^u ,\iota_q^{u,v} \}$, obtained as the $q$th homology of an $n$-filtration, we view it as the finitely generated $n$-graded $S$-module $V_q=\bigoplus_{u\in \N^n} V_q^u$ and denote its $i$th Betti table by $\xi_i^q$, for $i\in \{ 0,1, \ldots ,n\}$. We recall that its Betti tables are functions $\xi_i^q :\N^n \to \N$ defined by 
\[ \xi_i^q (u) \coloneqq \dim (\Tor_i^{S}(V_q,\F ))^u,
\]
for all $u\in \N^n$. 
Explicitly, $\xi^q_i (u)$ is the dimension (as an $\F$-vector space) of the piece of grade $u$ of the $n$-graded $S$-module $\Tor_i^{S}(V_q,\F )$. In Section \ref{sec:Koszul} we give an equivalent definition of the Betti tables based on the Koszul complex.

\subsection{Multigraded modules and free resolutions}
\label{subsec:resolutions}

We now briefly review free resolutions of $n$-graded modules over the polynomial ring $S:=\F [x_1, \ldots ,x_n]$. In this article, all $n$-graded $S$-modules are assumed to be finitely generated. Homomorphisms $f:V\to W$ between $n$-graded $S$-modules are assumed to be $n$-graded, meaning that they preserve grades: $f(V^u) \subseteq W^u$, for all $u\in \N^n$. We refer to \cite[Ch.\ 1]{miller2005combinatorial} and to texts like \cite{eisenbud2005geometry, peeva2010graded} for further details. 

For an $n$-graded $S$-module $V$ and for $a\in \Z^n$, we denote by $V(a)$ the module such that $V(a)^u=V^{u+a}$ for all $u\in \N^n$, called the {\em shift} of $V$ by $a$. The module $S(-a)$ is the  {\em free $S$-module on one generator} at grade $a\in \N^n$. It is isomorphic to the principal monomial ideal $\langle x^a \rangle$, where $x^a$ denotes the monomial $x_1^{a_1}\cdots x_n^{a_n}$. An $n$-graded $S$-module is called \emph{free} if it is isomorphic to $\bigoplus_{j=1}^r S(-a_j)$ for some $r\in \N$ and $a_j\in \N^n$. For a free module, $r$ and $\{a_j\}_{j=1}^r$ are uniquely determined. 

As an example related to the multifiltrations of Section \ref{subsec:multipers}, one can consider the persistence module $\{C_q(X^u),f_q^{u,v}\}_{u\preceq v \in \N^n}$, where the maps $f_q^{u,v}:C_q(X^u)\hookrightarrow C_q(X^v)$ are induced by the inclusions $X^u\hookrightarrow X^v$, and regard it as the $n$-graded $S$-module $C_q=\bigoplus_{u\in\N^n}C_q(X^u)$. If the $n$-filtration $\{ X^u \}_{u\in \N^n}$ is one-critical, then $C_q$ is free, isomorphic to $\bigoplus_{\sigma \in X_q} S(-v_{\sigma})$, where $v_{\sigma}$ denotes the unique entrance grade of the $q$-cell $\sigma$. The differential $\partial_q:C_q \to C_{q-1}$ is an example of an $n$-graded homomorphism between $n$-graded $S$-modules, whose component in grade $u$ is $\partial_q \colon C_q(X^u)\to C_{q-1}(X^u)$, for all $u\in\N^n$.

An ($n$-graded) \emph{free resolution} of an $n$-graded $S$-module $V$ is a sequence
\[
\cdots \to F_{\ell} \xrightarrow{\phi_{\ell}} F_{\ell-1} \xrightarrow{\phi_{\ell-1}} \cdots \xrightarrow{\phi_{2}} F_{1} \xrightarrow{\phi_{1}} F_{0} \to 0
\]
of $n$-graded free $S$-modules and $n$-graded homomorphisms which is exact at degree $i$ (that is, $\ker \phi_i =\im \phi_{i+1}$) for all $i>0$, and such that $\coker \phi_1 = V$. An exact sequence $\cdots \xrightarrow{\phi_{2}} F_{1} \xrightarrow{\phi_{1}} F_{0} \xrightarrow{\varepsilon} V \to 0$ is called an \emph{augmented} free resolution of $V$, with the $n$-graded homomorphism $\varepsilon$ called an \emph{augmentation}. The smallest integer $\ell$ (if it exists) for which $F_i = 0$ for every $i > \ell$ is called the length of the resolution. By Hilbert's Syzygy Theorem, every finitely generated $n$-graded $S$-module $V$ admits a free resolution with length $\ell \le n$. 

A free resolution is called \emph{minimal} if the image of each homomorphism $\phi_i$ is contained in $\langle x_1,\ldots ,x_n\rangle F_{i-1}$, where $\langle x_1,\ldots ,x_n\rangle$ denotes the homogeneous maximal ideal of $S$. 
Minimal free resolutions are unique up to isomorphism, and they are an invariant of the isomorphism type of $V$. In particular, the number of summands $S(-u)$ in $F_i$, for every $u\in \N^n$ and $i\in \{0,1,\ldots ,n\}$, is a well-defined invariant of $V$, and it coincides with the value at $u$ of the $i$th \emph{Betti table} (or \emph{multigraded Betti number}), $\xi_i (u) := \dim (\Tor_i^{S}(V,\F ))^u$. 
To see this, recall that, by definition, $\Tor_i^{S}(V,\F )$ can be determined by applying the functor $-\otimes_S \F$ to a free resolution of $V$ and taking the $i$th homology of the resulting chain complex of $n$-graded $S$-modules. Choosing a minimal free resolution of $V$, the homomorphisms $\phi_i \otimes_S \F$ are all zero, hence $\Tor_i^{S}(V,\F ) = F_i \otimes_S \F$ has in grade $u$ an $\F$-vector space of dimension equal to the number of summands $S(-u)$ in $F_i$, for all $u\in \N^n$.

A \emph{free presentation} of an $n$-graded $S$-module $V$ is an $n$-graded homomorphism $\phi_1 \colon F_1\to F_0$ between free $n$-graded $S$-modules $F_1$ and $F_0$ such that $\coker \phi_1=V$. In this article, we will occasionally refer to the augmented sequence $F_1 \xrightarrow{\phi_1} F_0 \to V \to 0$, which is exact at $F_0$ and $V$, as a free presentation of $V$. A free presentation of $V$ is called \emph{minimal} if it is the portion (in degrees 1 and 0) of a minimal free resolution of $V$.

\subsection{Discrete Morse theory and multifiltrations}
\label{subsec:DMT}

Discrete Morse theory, developed by Forman \cite{Forman1998}, is an adaptation of smooth Morse theory \cite{milnor1963morse} to a combinatorial framework. In its original formulation, it allows, given a regular CW complex, to construct a homotopy equivalent CW complex with a smaller number of cells. Building on Forman's work, discrete Morse theory has been formulated in purely algebraic terms for based chain complexes \cite{kozlov2005discrete} and in more general frameworks \cite{skoldberg2006combinatorial,jollenbeck2009resolution}. In this algebraic setting, the aim is to decompose a chain complex into a smaller complex and an acyclic complex. As explained in Section \ref{subsec:chaincm}, one can always take an equivalent combinatorial perspective by considering the cell complexes associated with based cell complexes. We briefly present here the main ideas of algebraic discrete Morse theory in the setting of this work.

Let $C_*(X)$ be the chain complex associated with a cell complex $X=\bigsqcup_q X_q$, 
and let $<$ be the covering face relation on $X$ introduced in Section \ref{subsec:chaincm}.
A pair of cells $(\sigma,\tau)\in X\times X$ with $\sigma < \tau$ is called a \emph{discrete vector}.
A \emph{discrete vector field} $\V$ on $X$ is a collection of discrete vectors $\V = \{ (\sigma_j, \tau_j)\}_{j\in J}$ such that all cells appearing in $\V$ (indifferently as the first or the second component of a vector) are different. A discrete vector field $\V$ determines a partition of $X$ into three graded subsets $M,S,T$, where $M$ is the set of unpaired cells, called \emph{critical cells}, and $S$ (respectively, $T$) is the set of cells appearing in $\V$ as first (respectively, second) components of a discrete vector. The subsets $M,S,T$ inherit the grading by dimension of the cells of $X$, so that for example $M=\bigsqcup_q M_q$. 
A $\V$-\emph{path} between two cells $\sigma$ and $\sigma'$ is a sequence $(\sigma_0,\tau_0,\sigma_1,\tau_1,\ldots ,\sigma_{r-1},\tau_{r-1},\sigma_r)$ with $r\ge 1$ such that $\sigma_0 = \sigma$, $\sigma_r =\sigma'$, each $(\sigma_i,\tau_i)$ is a discrete vector of $\V$, and  $ \sigma_{i+1} < \tau_i$. The $\V$-path is called \emph{closed} if $\sigma_r = \sigma_1$ and \emph{trivial} if $r=1$. A discrete vector field $\V$ is a \emph{discrete gradient vector field} (also called an \emph{acyclic matching} or a \emph{Morse matching}) when all closed $\V$-paths are trivial.

The core result of discrete Morse theory \cite{Forman1998} can be algebraically stated as follows \cite{KACZYNSKI199859,skoldberg2006combinatorial,jollenbeck2009resolution}.  

\begin{thm}
\label{thm:dfv-reduction}
Let $C_*(X)=(C_q(X), \partial_q)_{q \in \Z}$ be
the chain complex associated with a cell complex $X=\bigsqcup_q X_q$ and let $\V = \{(\sigma_j, \tau_j)\}_{j\in J}$ be a discrete gradient vector
field on $X$. Then $C_*(X)$ is chain homotopy equivalent to $C_*(M)=(C_q(M),\partial^M_q)_{q \in \Z}$, where $M=\bigsqcup_q M_q$ is the set of critical cells and  $\partial^M$ is a differential determined by $\partial$ and $\V$.
\end{thm}

We call $C_*(M)$ the \emph{(discrete) Morse chain complex} of $C_*(X)$ associated with $\V$. Let us stress that in general $C_*(M)$ is not a chain subcomplex of $C_*(X)$, since its differential $\partial^M$ is not simply induced by restriction by  the differential $\partial$ of $C_*(X)$. The details on how $\partial^M$ is (uniquely) determined by $\partial$ and $\V$ can be found in  \cite{skoldberg2006combinatorial,jollenbeck2009resolution}.
Equivalently, a cell complex structure on the set $M=\bigsqcup_q M_q$, called the \emph{(discrete) Morse complex} of $X$ associated with $\V$, is determined by the incidence function of $X$ and $\V$ \cite{KACZYNSKI199859}. In general, $M$ is not a subcomplex of $X$.

Discrete Morse theory of filtered chain complexes has been studied in a series of works related to one-parameter \cite{mischaikow2013morse} or multiparameter persistent homology \cite{allili2017reducing}. In the remainder of this subsection, we present the main ideas of discrete Morse theory for multifiltrations.

Consider an $n$-filtration $\{X^u\}_{u\in \N^n}$ of a cell complex $X$, which determines a filtration $\{C_* (X^u)\}_{u\in \N^n}$ of the chain complex $C_*(X)$. 
Given a discrete gradient vector field $\V$ on $X$, there are clearly induced filtrations $\{M^u\}_{u\in \N^n}$ on the Morse complex $M$ and $\{C_* (M^u)\}_{u\in \N^n}$ on the Morse chain complex $C_*(M)=(C_q(M),\partial^M_q)$. 
In general, the former is only a filtration of sets and the latter is only a filtration of graded $\F$-vector spaces, as the differential $\partial^M$ may fail to be compatible with the filtration. 
To avoid this, one can require the discrete gradient vector field to interact nicely with the multifiltration on $X$.
\begin{defn}
A discrete gradient vector field $\V$ on $X$ is \emph{consistent} with a multifiltration
$\{X^u\}_{u\in \N^n}$ if, for all
$(\sigma ,\tau)$ in $\V$ and all $u\in \N^n$,
$\sigma \in X^u$
if and only if
$\tau \in X^u$.
\end{defn}
If $\V$ is consistent with the multifiltration
$\{X^u\}_{u\in \N^n}$, then $\{C_* (M^u)\}_{u\in \N^n}$ is a filtration of chain subcomplexes of $C_*(M)$ \cite{mischaikow2013morse,allili2017reducing}. 
Equivalently, $\{M^u\}_{u\in \N^n}$ is a filtration of subcomplexes of $M$.
Moreover, the persistent homology modules associated with the multifiltrations of $X$ and its Morse complex are isomorphic (in the sense of Section \ref{subsec:multipers}).

\begin{prop}[Lemma~3.10 in \cite{allili2017reducing}]
\label{prop:iso_pers_mod}
Let $\V$ be a discrete gradient vector field on a cell complex $X$ consistent with an $n$-filtration
$\{X^u\}_{u\in \N^n}$, and let $\{M^u\}_{u\in \N^n}$ be the $n$-filtration induced on the Morse complex $M$.
Then, for any $q\in \N$, the persistence modules obtained as $q$-th homology of the $n$-filtrations $\{ X^u \}_{u\in \N^n}$ and $\{ M^u \}_{u\in \N^n}$ are isomorphic. 
\end{prop}

\section{Entrance grades and support of Betti tables via free resolutions}
\label{sec:free}

In this section, we illustrate how methods in multigraded homological algebra based on free presentations and resolutions (see Section \ref{subsec:resolutions}) can be used to derive relations between two different graded subsets of $\N^n$: the set of parameter grades at which new critical cells appear in the one-critical filtration $\{X^u\}_{u\in\N^n}$ of a cell complex $X$, on the one hand, and the set of parameter grades where the Betti tables of the persistent homology module $V_q=\bigoplus_{u\in\N^n}H_q(X^u)$ are nonzero, on the other hand. 
Specifically, we obtain bounds on the support of the $0$th and $1$st Betti table of $V_q$ (Proposition \ref{prop:suppfreepresVq}), and we discuss the immediate consequences of these bounds on the support of Betti tables of higher degrees (Remark \ref{rmk:free-bounds}). We conclude by observing that some of the stronger results we will prove in Section \ref{sec:support_xi} do not immediately follow from this approach. For this reason, we defer the discussion of how our results on the support of Betti table can be combined with discrete Morse theory to Section \ref{subsec:Morse-supp}.

In this section we consider  the following setting.
Let $\{X^u\}_{u\in\N^n}$ be a one-critical $n$-filtration of a cell complex $X$. 
We assume the multifiltration $\{ X^u \}_{u\in \N^n}$ to be \emph{exhaustive}, that is, $X=\bigcup_{u\in \N^n}X^u$. Clearly, since $X$ is graded by the dimension $q$ of cells, this means that $X_q = \bigcup_{u\in \N^n} X^u_q$, for all $q\in \N$. The one-criticality assumption (Section \ref{subsec:multipers}) ensures that the chain complex associated with the filtration $\{X^u\}_{u\in \N^n}$ is made of free $n$-graded modules over the polynomial ring $S:=\F [x_1, \ldots ,x_n]$. More specifically,
for any $q\in\N$, the $n$-graded $S$-module $C_q \coloneqq \bigoplus_{u\in \N^n}C_q(X^u)$ associated with the filtration is free and isomorphic to $\bigoplus_{\sigma \in X_q} S(-v_{\sigma})$, where $v_{\sigma}$ denotes the unique entrance grade of the $q$-cell $\sigma$. The set of all entrance grades (Section \ref{subsec:multipers}) of $q$-cells is denoted by $\G(X_q) \subseteq \N^n$, and its closure with respect to least upper bounds is denoted by $\overline{\G(X_q)}$. Explicitly, $\overline{\G(X_q)}\coloneqq \{\bigvee L \mid L\subseteq \G(X_q), L\ne \emptyset \}\subseteq \N^n$, with $\bigvee L$ denoting the least upper bound of $L$ in $(\N^n,\preceq)$.

By definition, the persistent homology module $V_q=\bigoplus_{u\in\N^n}H_q(X^u)$ is the homology at the middle term of the sequence $C_{q+1}\xrightarrow{\partial_{q+1}}C_{q}\xrightarrow{\; \partial_{q}\;}C_{q-1}$ of free $n$-graded $S$-modules and $n$-graded homomorphisms.
Our aim is constructing a free resolution of $V_q$ that is informative of the relation between the support of the Betti tables and the sets of entrance grades of cells. In this section, we denote by $\xi_i(V)$ the $i$th Betti table of an $n$-graded $S$-module $V$, which we view as a function $\xi_i(V)\colon \N^n \to \N$ with values $\xi_i(V)(u)\coloneqq \dim (\Tor_i^{S}(V,\F ))^u$ defined as detailed in Section \ref{subsec:resolutions}. We drop the module $V$ from the notation of the Betti tables when it is clear from the context. Lastly, let us recall that we use the notation $\xi_i^q \coloneqq \xi_i(V_q)$ for the Betti tables of the persistent homology module $V_q=\bigoplus_{u\in\N^n}H_q(X^u)$, and that we denote by $\supp \xi^q_i := \{ u \in \N^n \mid \xi^q_i (u) \ne 0 \}$ the support of $\xi^q_i$.

We start by considering the following sequence of $n$-graded $S$-modules and $n$-graded homomorphisms,
\[
\begin{tikzcd}
	{C_{q+1}} & {\ker \partial_q} & {V_q=\frac{\ker \partial_q}{\im \partial_{q+1}}}
	\arrow["{\partial_{q+1}}", from=1-1, to=1-2]
	\arrow["h", two heads, from=1-2, to=1-3] ,
\end{tikzcd}
\]
where $h$ is the canonical projection.
This sequence is not a free presentation of $V_q$ in general, since $\ker \partial_q$ is in general not free for $n>2$ and $q>0$. To obtain a free presentation of $V_q$, we consider a free presentation $F_1 \to F_0 \twoheadrightarrow \ker \partial_q$ of $\ker \partial_q$, which we assume to be minimal. This free presentation is the second row in the diagram of $n$-graded $S$-modules
\[
\begin{tikzcd}
	0 & {C_{q+1}} & {C_{q+1}} \\
	{F_1} & {F_0} & {\ker \partial_q} \\
	&& {V_q}
	\arrow[from=1-1, to=1-2]
	\arrow[from=1-1, to=2-1]
	\arrow[Rightarrow, no head, from=1-2, to=1-3]
	\arrow["{\overline{\partial}}", from=1-2, to=2-2]
	\arrow["{\partial_{q+1}}", from=1-3, to=2-3]
	\arrow["{\phi_1}", from=2-1, to=2-2]
	\arrow["\varepsilon", two heads, from=2-2, to=2-3]
	\arrow["h", two heads, from=2-3, to=3-3]
\end{tikzcd}
\]
where the homomorphism $\overline{\partial}$ is a lift of $\partial_{q+1}$, which exists since $C_{q+1}$ is free (hence projective) and $\varepsilon$ is surjective. A free presentation of $V_q$ is then given by
\begin{equation}\label{eq:freepresVqq}
\begin{tikzcd}[column sep=large] 
	{C_{q+1}\oplus F_1} & {F_0} & {V_q}
	\arrow["{[\overline{\partial} \; \phi_1]}", from=1-1, to=1-2]
	\arrow["h\varepsilon",twoheadrightarrow, from=1-2, to=1-3],
\end{tikzcd}
\end{equation}
where $[\overline{\partial} \; \phi_1]$ denotes the $n$-graded homomorphism sending $(c,x)\in C_{q+1}\oplus F_1$ to $\overline{\partial}(c) + \phi_1(x) \in F_0$. To see that $\coker [\overline{\partial} \; \phi_1] = V_q$, we observe that the composition $h\varepsilon$ is surjective, and that its kernel coincides with $\im [\overline{\partial} \; \phi_1] = \im \overline{\partial} + \im \phi_1$. 

Our goal is approximating the sets of grades of the generators of the free modules $F_0$ and $C_{q+1}\oplus F_1$, which are the sets $\supp \xi_0 (F_0)$ and $\supp \xi_0 (C_{q+1}\oplus F_1)$, respectively. In Proposition \ref{prop:suppfreepresVq} we state bounds in terms of the sets $\G(X_{q})$ and $\G(X_{q+1})$. To prove these bounds, we need some results on free resolution of $n$-graded $S$-modules. 

First, we state a result whose proof can be found for example in \cite[Lemma 2.1]{vipond2020local} or, in a slightly different setting, in \cite[Corollary~4.2]{chacholski2024koszul}.

\begin{prop}\label{prop:xifreeres}
Let $V$ be a (finitely generated) $n$-graded $S$-module. Then the supports of its Betti tables satisfy the containments $\supp \xi_{i+1} \subseteq \overline{\supp \xi_{i}}$, for all $i\ge 1$.
\end{prop}

Next, we need a result on the structure of free resolutions. The proofs presented in \cite[Theorem~7.5]{peeva2010graded} or \cite[p.~6]{eisenbud2005geometry} carry over to the multigraded case.
\begin{prop}\label{prop:minsummand}
Every $n$-graded free resolution of an $n$-graded $S$-module $V$ is isomorphic to the direct sum of a minimal free resolution of $V$ and short trivial complexes of the form $0\to S(-u)\xrightarrow{\id} S(-u)\to 0$, with $u\in\N^n$, possibly involving different homological degrees.
\end{prop}

The following is a useful consequence of Propositions~\ref{prop:xifreeres} and \ref{prop:minsummand}.

\begin{coro}\label{coro:xiker}
Let $K$ be the kernel of an $n$-graded homomorphism $f:V\to W$ of $n$-graded $S$-modules, where $V$ is free. Then $\supp \xi_1(K) \subseteq \overline{\supp \xi_0 (K)}$.     
\end{coro}
\begin{proof}
Let $F_*=(\cdots \to F_1\xrightarrow{\phi_1} F_0 \to 0)$ be an $n$-graded minimal free resolution of $K$. The augmented free resolution $\cdots \to F_1\xrightarrow{\phi_1} F_0 \xrightarrow{\varepsilon} K \to 0$
can be composed with the canonical monomorphism $K \xhookrightarrow{\iota} V$ to form the sequence
\begin{equation}\label{eq:freeresimf}
\cdots \to F_1\xrightarrow{\phi_1} F_0 \xrightarrow{\iota \varepsilon} V\to 0
\end{equation}
of free $n$-graded $S$-modules, which can be viewed as a (non-necessarily minimal) free resolution of the module $\im f\cong V/K\cong \coker \iota\varepsilon $. By Proposition \ref{prop:minsummand}, the free resolution (\ref{eq:freeresimf}) of $\im f$ is isomorphic to a minimal free resolution $P_* = (\cdots \to P_2\to P_1\to P_0 \to 0)$ plus a direct sum of short trivial complexes. By minimality of $F_*$, a short trivial complex $0\to S(-u)\xrightarrow{\id} S(-u)\to 0$ which is a direct summand of (\ref{eq:freeresimf}) can only have nonzero modules in homological degrees $i=0,1$ (using indices as in $P_*$). Since Betti tables count the multiplicity of free summands $S(-u)$ at each grade $u\in \N^n$ and each homological degree of a minimal free resolution (see Section \ref{subsec:resolutions}), this implies that $\supp \xi_1(\im f)\subseteq \supp \xi_0 (K)$ and that $\supp \xi_{i+1}(\im f)= \supp \xi_{i} (K)$ for $i\ge 1$, which together with Proposition \ref{prop:xifreeres} gives
\[
\supp \xi_1(K) = \supp \xi_2 (\im f) \subseteq \overline{\supp \xi_1 (\im f)} \subseteq \overline{\supp \xi_0 (K)}.
\]
\end{proof}

We are now ready to prove bounds for the grades of the generators of the free modules appearing in the free presentation (\ref{eq:freepresVqq}). 

\begin{prop}\label{prop:suppfreepresVq}
The containments $\supp \xi_0 (F_0) \subseteq \overline{\G (X_q)}$ and $\supp \xi_0 (C_{q+1}\oplus F_1) \subseteq \G (X_{q+1}) \cup \overline{\G (X_q)}$ hold for the modules in the free presentation (\ref{eq:freepresVqq}) of $V_q$.
\end{prop}
\begin{proof}
We start with an argument similar to the one used in the proof of Corollary \ref{coro:xiker}.
Let $F_*=(\cdots \to F_1\xrightarrow{\phi_1} F_0 \to 0)$ be a minimal free resolution of $\ker \partial_q$. The augmented exact sequence $\cdots \to F_1\xrightarrow{\phi_1} F_0 \xrightarrow{\varepsilon} \ker \partial_q \to 0$
can be spliced with the exact sequence $0\to \ker \partial_q \xrightarrow{\iota} C_q\xrightarrow{\partial_q} C_{q+1}$ to form the sequence 
\begin{equation}\label{eq:freerescoker}
\cdots \to F_1\xrightarrow{\phi_1} F_0 \xrightarrow{\iota \varepsilon} C_q\xrightarrow{\partial_q} C_{q+1} \to 0
\end{equation}
of free $n$-graded $S$-modules, which can be viewed as a (non-necessarily minimal) free resolution of the module $\coker \partial_q$. By Proposition \ref{prop:minsummand}, the free resolution (\ref{eq:freerescoker}) is isomorphic to a minimal free resolution $P_* = (\cdots \to P_3\to P_2\to P_1\to P_0 \to 0)$ plus a direct sum of short trivial complexes. We observe that a short trivial complex $0\to S(-u)\xrightarrow{\id} S(-u)\to 0$ with nonzero modules in homological degrees $i=2,3$ cannot be a direct summand of (\ref{eq:freerescoker}), by minimality of the free resolution $F_*$ of $\ker \partial_q$. For this reason, the containment $\supp \xi_0(P_2)\subseteq \overline{\supp \xi_0 (P_1)}$, obtained by applying Proposition \ref{prop:xifreeres} to $P_*$ with $i=1$, implies the containment $\supp \xi_0(F_0)\subseteq \overline{\supp \xi_0 (C_q)}$. The first containment of the claim follows by recalling that the set $\supp \xi_0 (C_q)$ of grades of the generators of $C_q$ coincides with $\G(X_q)$ by definition.

We now consider the set $\supp \xi_0 (C_{q+1}\oplus F_1)=\supp \xi_0 (C_{q+1})\cup \supp \xi_0 (F_1)$. Again by definition, we have $\supp \xi_0 (C_{q+1})=\G(X_{q+1})$. We therefore focus on the set $\supp \xi_0 (F_1)$ and observe that
\[
\supp \xi_0 (F_1)=\supp \xi_1(\ker \partial_q) \subseteq \overline{\supp \xi_0(\ker \partial_q)},
\]
where the equality is by definition of Betti tables via minimal resolutions (Section \ref{subsec:resolutions}), and the containment is by Corollary \ref{coro:xiker}. The second containment of the claim then follows from the equality $\xi_0 (\ker \partial_q)=\xi_0 (F_0)$ and from the first part of the proof.
\end{proof}
\begin{rmk}\label{rmk:free-bounds}
Since by Proposition \ref{prop:minsummand} the free presentation (\ref{eq:freepresVqq}) of $V_q$ contains a minimal free presentation as a direct summand, Proposition \ref{prop:suppfreepresVq} yields the following two containments: 
\[
\supp \xi_0^q \subseteq \overline{\G(X_q)}, \qquad \supp \xi_1^q \subseteq \G(X_{q+1})\cup  \overline{\G(X_q)}.
\]
We recall that the support of the $1$st Betti table determines a bound for the support of all Betti tables of positive degree, since $\bigcup_{i=1}^n \supp \xi_i^q \subseteq \overline{\supp \xi_1^q}$. This general fact for $n$-graded $S$-modules is observed for example in \cite[Remark~3.2]{charalambous2015betti}, and follows from Proposition \ref{prop:xifreeres}. Using this fact, we immediately see that
\begin{equation*}\label{eq:bound-by-res}
\bigcup_{i=1}^n \supp \xi^q_i \subseteq \overline{\G (X_{q+1}) \cup \overline{\G (X_q)}}.
\end{equation*}
\end{rmk}

In Section \ref{sec:support_xi}, we will obtain the containments for $\supp \xi_0^q$ (Theorem \ref{thm:supp_xi_m}) and $\supp \xi_1^q$ (Corollary \ref{coro:H1K}) with an alternative method based on the Koszul complex, which will allow us to improve some of the statements regarding the support of higher Betti tables.
In Theorem \ref{thm:supp_xi_m}, we will prove the stronger statement $\bigcup_{i=0}^n \supp \xi^q_i   \subseteq \overline{\G (X_{q+1})} \cup \overline{\G (X_{q})}$, together with the containment $\supp \xi^q_n   \subseteq \overline{\G (X_{q+1})}$ for the support of the $n$th Betti table.

\section{The Koszul complex of a persistence module}
\label{sec:Koszul}

In this section, we describe the Koszul complex associated with an $n$-parameter persistence module and illustrate some of its properties. 
In particular, given an $n$-parameter persistent homology module $\{ H_q (X^u), \iota_q^{u,v}\}$, we introduce its Koszul complex at $u\in \N^n$, a chain complex whose $i$th homology module has dimension equal to the Betti table value $\xi^q_i (u)$. This chain complex can be constructed via a repeated procedure which allows us to add one parameter of the multifiltration at a time.

In Section \ref{subsec:Kosz1}, upon briefly recalling general definitions and results, we provide a more detailed description of Koszul complexes of multiparameter persistent homology modules. We claim no original results in this subsection, as the Koszul complex is a standard tool, and the explicit description of its chain modules and differentials in the case of persistent homology modules is included, for example, in \cite[Sect.~3]{guidolin2023morse}. Here, besides fixing notations, we provide further details, especially with regard to bifiltrations, that are relevant to this work. 

In Section \ref{sec:mappingcones}, we explain how the Koszul complex associated with an $n$-parameter persistence module can be constructed as an iterated mapping cone, and we highlight the role of this construction for persistent homology modules, which intuitively allows one to disentangle the different parameters of the multifiltration and study their impact on the Betti tables. In Section \ref{sec:support_xi}, we will apply this technique to study the support of the Betti tables.

\subsection{The Koszul complex of a multigraded module}
\label{subsec:Kosz1}

Let $S$ denote the polynomial ring $\F[x_1,\ldots ,x_n]$. We recall that, for any subset $\alpha \subseteq [n]$, we set $e_{\alpha}\= \sum_{j\in \alpha} e_j \in \{0,1\}^n$. The \emph{Koszul complex} $\mathbb{K}_*$ is a  chain complex of free $n$-graded $S$-modules whose construction is standard in commutative algebra (cf.\ \cite[Def.~1.26]{miller2005combinatorial}): 
for each $i$, let $\mathbb{K}_{i} \= \bigoplus_{\alpha \subseteq [n], \; |\alpha|=i} S(-e_{\alpha})$, where $S(-e_{\alpha})$ denotes the free $S$-module generated in grade $e_{\alpha}$ by an element we denote $1_{\alpha}$, for some $\alpha = \{ j_1 < j_2 <\ldots < j_i \}$. The differentials $d^K_i : \mathbb{K}_{i} \to \mathbb{K}_{i-1}$ are defined on generators by
\begin{equation*} 
\label{eq:dKos} 
    d^K_i (1_{\alpha}) = \sum_{r=0}^{i-1} (-1)^{r} x_{j_{i-r}} \cdot 1_{\alpha \smallsetminus \{ j_{i-r} \}} .
\end{equation*} 
Given an $n$-graded $S$-module $V=\bigoplus_{u\in \N^n}V^u$, the \emph{Koszul complex} $\mathbb{K}_*(x_1,\ldots ,x_n; V)(u)$ \emph{of $V$ at grade $u\in\N^n$} is the piece of grade $u$ of the ($n$-graded) chain complex $V \otimes_S \mathbb{K}_*$. 
This chain complex of $\F$-vector spaces can be used to determine the Betti tables $\xi_i (u) := \dim (\Tor_i^{S}(V,\F ))^u$ of $V$ at grade $u$, for $i\in \{0,1,\ldots ,n\}$.
Indeed, by definition, $\Tor_i^{S}(V,\F )$ can be determined by applying the functor $-\otimes_S \F$ to a free resolution of $V$ and taking $i$th homology of the resulting chain complex (see Section \ref{subsec:resolutions}). 
The roles of $V$ and $\F$ can however be interchanged, by virtue of the isomorphism $\Tor_i^{S}(V,\F ) \cong \Tor_i^{S}(\F, V)$  (see, e.g., \cite[Thm.~7.1]{rotman2009homological}); in this case, choosing $\mathbb{K}_*$ as a (minimal) free resolution of $\F$ (see \cite[Prop.~1.28]{miller2005combinatorial}) yields, for all $i\in \{0,1,\ldots ,n\}$, the equality
\[
\xi_i (u) = \dim H_i (\mathbb{K}_*(x_1,\ldots ,x_n; V)(u)) . 
\]

Let us now provide a more explicit description of the Koszul complex $\mathbb{K}_*(x_1,\ldots ,x_n; V_q)(u)$ of a persistent homology module $\{ H_q(X^u), \iota_q^{u,v} \}$ associated with an $n$-parameter filtration $\{ X^u\}_{u\in \N^n}$, regarded as an $n$-graded $S$-module $V_q = \bigoplus_{u\in \N^n}H_q (X^u)$ (as reviewed in Section \ref{subsec:multipers}). 
Even if this description of the Koszul complex can be easily adapted to any $n$-parameter persistence module, not necessarily built from a filtered cell complex, we prefer to focus on the case of interest for this work in order to clearly introduce the notations we will use in what follows.

For each $i\in \{0,1,\ldots , n \}$, the chain module in degree $i$ of  $\mathbb{K}_*(x_1,\ldots ,x_n; V_q)(u)$ is
\[
\mathbb{K}_i (x_1,\ldots ,x_n; V_q)(u) = \bigoplus_{\alpha \subseteq [n], \; |\alpha|=i} H_q (X^{u-e_{\alpha}}) .
\] 
The definition can be easily extended if, for some fixed $u\in \N^n$ and some $\alpha \subseteq [n]$, it happens that $u-e_{\alpha} \notin \N^n$: throughout this article, by definition, we set $X^w=\emptyset$ whenever the grade $w$ is not in $\N^n$. 
Note that the modules $\mathbb{K}_i (x_1,\ldots ,x_n; V_q)(u)$ are zero for all $i\notin \{0,1,\ldots , n \}$.
The differentials of $\mathbb{K}_* (x_1,\ldots ,x_n; V_q)(u)$ are defined in terms of the maps $\iota^{v,w}_{q}: H_q (X^v) \to H_q (X^w)$ as follows: the differential
\[
d_i : \mathbb{K}_i (x_1,\ldots ,x_n; V_q)(u) \to \mathbb{K}_{i-1} (x_1,\ldots ,x_n; V_q)(u)
\]
is defined as the alternating sum $d_{i} = \sum_{r=0}^{i-1} (-1)^r d_{i,r}$ of functions  $d_{i,r} \colon \mathbb{K}_i (x_1,\ldots ,x_n; V_q)(u) \to \mathbb{K}_{i-1} (x_1,\ldots ,x_n; V_q)(u)$  mapping the summand $H_q(X^{u-e_{\alpha}})$ in $\mathbb{K}_i (x_1,\ldots ,x_n; V_q)(u)$, with $\alpha = \{ j_1 < j_2 <\ldots < j_i \}$, to the summand $H_q(X^{u-e_{\alpha}+e_{j_{i-r}}})$ in $\mathbb{K}_{i-1} (x_1,\ldots ,x_n; V_q)(u)$, 
via the function $\iota_q^{u-e_{\alpha}, \, u-e_{\alpha}+e_{j_{i-r}}} $.
For the sake of a simpler notation, we avoid denoting the grade $u$ in the differentials $d_i$. 
As we explained, $\xi^q_i(u)$ coincides with the dimension (as an $\F$-vector space) of the $i$th homology module of $\mathbb{K}_*(x_1,\ldots ,x_n; V_q)(u)$.

Let us detail the cases of $n=1$ and $n=2$ parameters for later convenience. 
For a $1$-parameter filtration $\{ X^u\}_{u\in \N}$, the Koszul complex $\mathbb{K}_* (x_1; V_q)(u)$ of $V_q = \bigoplus_{u\in \N}H_q (X^u)$ at $u\in \N$ is
\[
0 \xrightarrow{} H_q(X^{u-1}) \xrightarrow{d_1 = \iota_q^{u-1,u}} H_q(X^{u}) \xrightarrow{} 0 .
\]
The Betti tables at grade $u$ are $\xi_0^q (u) = \dim \coker \iota_q^{u-1,u}$ and $\xi_1^q (u) =  \dim \ker \iota_q^{u-1,u}$, which correspond respectively to the number of \emph{births} and \emph{deaths} of $q$-homology classes at $u\in \N$ in the sense of persistence \cite{Edelsbrunner2002}. 

For a $2$-parameter filtration $\{ X^u\}_{u\in \N^2}$, the Koszul complex $\mathbb{K}_* (x_1,x_2; V_q)(u)$ of the module $V_q = \bigoplus_{u\in \N^2}H_q (X^u)$ at $u\in \N^2$ is  
\begin{equation*}
\label{eq:Kosz1}
0 \xrightarrow{} H_q(X^{u-e_1 -e_2})
\xrightarrow{d_{2}} H_q(X^{u-e_{1}})\oplus H_q(X^{u-e_{2}}) \xrightarrow{d_{1}}  H_q(X^{u}) \xrightarrow{}  0 ,
\end{equation*}
with differentials
\begin{align*}
d_2 &= \begin{bmatrix}
 - \iota_q^{u-e_1-e_2,u-e_1} \\
  \iota_q^{u-e_1-e_2,u-e_2}
\end{bmatrix} 
\qquad \text{and} \qquad 
d_1 = [\iota_q^{u-e_1,u} \quad \iota_q^{u-e_2,u}]  .
\end{align*}
The Betti tables at the grade $u$ are
\begin{equation*}  
\xi_2^q(u) =  \dim \ker d_2 ,  \qquad
\xi_1^q(u) =  \dim (\ker d_1 / \im d_2) ,  \qquad
\xi_0^q(u) =  \dim \coker d_1 .
\end{equation*}

A morphism $\nu = \{ \nu^u : V^u \to W^u \}_{u\in \N^n}$ between $n$-parameter persistence modules $\{ V^u ,\varphi^{u,v} \}$ and $\{ W^u ,\psi
^{u,v} \}$ induces a chain map between the Koszul complexes of $V=\bigoplus_{u\in \N^n}V^u$ and $W=\bigoplus_{u\in \N^n}W^u$ at $u\in \N^n$, the morphism between the chain modules in degree $i$ being $\bigoplus_{|\alpha|=i} \nu^{u-e_{\alpha}}$, with $\alpha \subseteq [n]$. 
Moreover, since taking finite direct sums preserves short exact sequences of vector spaces, taking the Koszul complex at any fixed $u$ is an exact functor, meaning that a short exact sequence $0\to U \xrightarrow{\mu} V \xrightarrow{\nu} W \to 0$ of $n$-parameter persistence modules induces a short exact sequence of Koszul complexes
\[
0\to \mathbb{K}_*(x_1,\ldots ,x_n; U)(u) \to \mathbb{K}_*(x_1,\ldots ,x_n; V)(u) \to \mathbb{K}_*(x_1,\ldots ,x_n; W)(u) \to 0 .
\]

Clearly, an isomorphism between persistence modules induces an isomorphism between their Koszul complexes. In what follows, we will apply this observation in the particular case of a multifiltration $\{ X^u \}_{u\in \N^n}$ of $X$ and the induced multifiltration $\{ M^u \}_{u\in \N^n}$ of its Morse complex.
By virtue of Proposition \ref{prop:iso_pers_mod}, since the modules $V_q := \bigoplus_{u\in \N^n}H_q (X^u)$ and $V'_q := \bigoplus_{u\in \N^n}H_q (M^u)$ are isomorphic, their Koszul complexes $\mathbb{K}_*(x_1,\ldots ,x_n; V_q)(u)$ and $\mathbb{K}_*(x_1,\ldots ,x_n; V'_q)(u)$ are also isomorphic, at all $u\in \N^n$. As a consequence, the Betti tables $\xi_i^q(u)$ can be determined considering the Morse complex instead of the original complex.

\subsection{Explicit construction via mapping cones}
\label{sec:mappingcones}

We now illustrate the explicit construction of the Koszul complex $\mathbb{K}_* (x_1,\ldots ,x_n; V_q)(u)$ of $V_q$ at grade $u\in \N^n$ as an iterated mapping cone. The classical construction of the Koszul complex via mapping cones can be found in \cite[\S~A2F]{eisenbud2005geometry} and \cite[Ch.~1.6]{bruns1998cohen}; here we rephrase, adapt, and enrich it with examples, to provide a complete and explicit treatment for Koszul complexes of persistent homology modules that conveys the intuition of persistent homology.

Given a chain map $f: B_* \to C_*$, the \emph{mapping cone} $\cone (f)_*$ of $f$ is the chain complex with $\cone (f)_i \= B_{i-1}\oplus C_i$ and differential $\delta_i : B_{i-1}\oplus C_i \to B_{i-2}\oplus C_{i-1}$ defined by
\begin{equation}
\label{eq:def_diffCone}
\delta_i (b,c) \= (-\partial^B_{i-1}(b), \; \partial^C_i (c) + f_{i-1}(b)),
\end{equation}
for all $i$, with $b\in B_{i-1}, c\in C_{i}$ and $\partial^B, \partial^C$ respectively denoting the differentials of $B_*$ and $C_*$, see \cite[\S~A3.12]{eisenbud1995commutative}. 

Let $\mathcal{F}:=\{ X^u\}_{u\in \N^n}$ be an $n$-filtration of a cell complex $X$.
As is evident from the definitions in Section \ref{subsec:Kosz1}, the
Koszul complex $\mathbb{K}_*(x_1,\ldots ,x_n; V_q)(u)$ of the associated persistent homology module $V_q = \bigoplus_{u\in \N^n}H_q (X^u)$ at the fixed grade $u\in \N^n$ only depends on the subcomplexes $X^{u-e_{\alpha}}$ of the filtration, with $\alpha \subseteq [n]$. 
In other words, to determine $\mathbb{K}_*(x_1,\ldots ,x_n; V_q)(u)$ it is enough to consider the smaller $n$-filtration $\mathcal{F}^u :=\{X^{u-e_{\alpha}}\}_{\alpha \subseteq [n]}$, containing $2^{n}$ subcomplexes of the original $n$-filtration $\mathcal{F}$. 
We observe that, fixed any $j\in [n]$, the $n$-filtration $\mathcal{F}^u$ can be partitioned into $2^{n-1}$  $1$-filtrations $X^{u-e_{\alpha}-e_j} \subseteq X^{u-e_{\alpha}}$, one for each $\alpha \subseteq [n] \smallsetminus \{j\}$. 
More generally, fixed any non-empty subset $J\=\{j_1,\ldots ,j_t\} \subseteq [n]$, there is a partition of $\mathcal{F}^u$ consisting of $2^{n-t}$ $t$-filtrations of the form $\{X^{u-e_{\alpha}-e_{\gamma}}\}_{\gamma \subseteq J}$, one for each $\alpha \subseteq [n]\smallsetminus J$. 
Every such $t$-filtration has an associated Koszul complex $\mathbb{K}_*(x_{j_1},\ldots ,x_{j_t}; V_q)(u-e_{\alpha})$ that intuitively only encodes information  on the parameters $j_1,\ldots ,j_t$ of the $n$-filtration $\mathcal{F}^u$. 
Given $k \in [n]\smallsetminus J$, regarded here as an additional parameter  to be taken into account, one can consider the $(t+1)$-filtration given by the union of two $t$-filtrations $\{X^{u-e_{\alpha}-e_{\gamma}}\}_{\gamma \subseteq J}$ and $\{X^{u-e_{\alpha}-e_k -e_{\gamma}}\}_{\gamma \subseteq J}$, for any $\alpha \subseteq [n]\smallsetminus (J\cup \{k\})$. Below, we will explain how the Koszul complex associated with such $(t+1)$-filtration can be constructed as the mapping cone of a chain map between the two Koszul complexes associated with the $t$-filtrations.

We begin by illustrating in detail the first few steps of the procedure based on iterated mapping cones to construct the Koszul complex $\mathbb{K}_* (x_1,\ldots ,x_n; V_q)(u)$ starting from ``$1$-parameter'' Koszul complexes ``in direction  $e_j$''
\begin{equation*}
\mathbb{K}_* (x_j;V_q)(w) = \left( 0 \xrightarrow{} H_q(X^{w-e_j}) \xrightarrow{d_1 = \iota_q^{w-e_j,w}} H_q(X^{w}) \xrightarrow{} 0 \right) ,
\end{equation*}
for any fixed $j\in [n]$ and for $w = u -e_{\alpha}$ with $\alpha \subseteq [n]\smallsetminus \{ j\}$, and from specific chain maps between them. The chain maps are those induced by inclusions ``in direction $e_k$'', for any fixed $k\in [n]\setminus \{j \}$, that is 
\[
f^k(x_{j}; V_q)(w-e_k) : \mathbb{K}_* (x_{j}; V_q)(w-e_k) \to \mathbb{K}_* (x_{j}; V_q)(w) ,
\]
with $f_i^k(x_{j}; V_q)(w-e_k) : \mathbb{K}_i (x_{j}; V_q)(w-e_k) \to \mathbb{K}_i (x_{j}; V_q)(w)$ defined, for degrees $i=0,1$, as
\begin{align*} 
f_0^k(x_{j}; V_q)(w-e_k) &= \iota_q^{w-e_k, w} : H_q(X^{w-e_k}) \xrightarrow{} H_q(X^{w}) , \\
f_1^k(x_{j}; V_q)(w-e_k) &= \iota_q^{w-e_{j}-e_k, w-e_{j}} : H_q(X^{w-e_j-e_k})\xrightarrow{} H_q(X^{w-e_j}) .
\end{align*}
The mapping cone $\cone (f^k(x_{j}; V_q)(w-e_k))_*$ is the Koszul complex $\mathbb{K}_* (x_{j},x_k; V_q)(w)$, 
associated with the $2$-filtration $\{X^{w-e_{\gamma}}\}_{\gamma \subseteq \{j,k\}}$. Intuitively, it is obtained from the previous step, where only the $j$th parameter was considered,  by adding one parameter more, namely the $k$th parameter of the original $n$-filtration.
Explicitly, $\mathbb{K}_* (x_{j},x_k; V_q)(w)$ is the chain complex
\begin{equation*}
0 \xrightarrow{} H_q(X^{w-e_j -e_k}) \xrightarrow{\; d_2 \;} H_q(X^{w-e_j}) \oplus H_q(X^{w-e_k})  \xrightarrow{\; d_1 \;} H_q(X^{w})
\xrightarrow{} 0
\end{equation*}
where the differentials, applying the definition (\ref{eq:def_diffCone}), are 
\begin{align*}
d_2 &= \begin{bmatrix}
           - \iota_q^{w-e_j-e_k,w-e_j} \\
           \iota_q^{w-e_j-e_k,w-e_k}
         \end{bmatrix} 
\qquad \text{and} \qquad  
d_1  = [\iota_q^{w-e_j,w} \quad \iota_q^{w-e_k,w}] .
\end{align*}

The process we just described can be repeated, by choosing a new ``direction'' $e_\ell$ corresponding to a new parameter $\ell \in [n] \smallsetminus \{j,k\}$ and constructing $\mathbb{K}_* (x_{j},x_k,x_\ell; V_q)(w)$ as the mapping cone of the chain map $f^{\ell}(x_{j},x_{k}; V_q)(w-e_\ell)$ induced by inclusions in direction $e_\ell$, for each $w=u-e_{\alpha}$ with $\alpha \subseteq [n]\smallsetminus \{ j,k,\ell \}$. 
Explicitly, $f^{\ell}(x_{j},x_{k}; V_q)(w-e_\ell)$ is defined by the following maps, in degrees $i=0,1,2$:
\begin{align*} 
f_0^{\ell}(x_{j},x_{k}; V_q)(w-e_\ell) &= \iota_q^{w-e_\ell, w}, \\
f_1^{\ell}(x_{j},x_{k}; V_q)(w-e_\ell) &= \iota_q^{w-e_j -e_\ell, w-e_j}\oplus \iota_q^{w-e_k -e_\ell, w-e_k}, \\
f_2^{\ell}(x_{j},x_{k}; V_q)(w-e_\ell) &= \iota_q^{w-e_j -e_k-e_\ell, w-e_j -e_k}.
\end{align*}
If the order in which the indeterminates are added is changed, one obtains isomorphic chain complexes: for example, $\mathbb{K}_* (x_{j},x_k,x_\ell; V_q)(w)$ is isomorphic to $\mathbb{K}_* (x_{j},x_\ell,x_k; V_q)(w)$. At the last step, one obtains $\mathbb{K}_* (x_1,\ldots,x_n; V_q)(u)$ as the mapping cone of the chain map
$f^m(x_1 , \ldots , \hat{x}_m, \ldots , x_n; V_q)(u-e_m)$ between $\mathbb{K}_* (x_1 , \ldots , \hat{x}_m, \ldots , x_n; V_q)(u-e_m)$ and $\mathbb{K}_* (x_1 , \ldots , \hat{x}_m, \ldots , x_n; V_q)(u)$.

Thanks to the iterative nature of the process, we can provide 
an explicit description of $\mathbb{K}_* (x_{j_1},\ldots ,x_{j_t}; V_q)(u)$ for any $u\in \N^n$ and any non-empty subset $J\=\{j_1,\ldots ,j_t\} \subseteq [n]$. 
For each $i\in \{0,1,\ldots ,|J| \}$, the chain module in degree $i$ is
\[
\mathbb{K}_i (x_{j_1},\ldots ,x_{j_t}; V_q)(u) = \bigoplus_{\gamma \subseteq J, \; |\gamma|=i} H_q (X^{u-e_{\gamma}}).
\]
The modules $\mathbb{K}_i (x_{j_1},\ldots ,x_{j_t}; V_q)(u)$ are zero for all $i\notin \{0,1,\ldots , |J| \}$.
The differentials of the chain complex $\mathbb{K}_* (x_{j_1},\ldots ,x_{j_t}; V_q)(u)$ can be described as follows: the differential 
\[
d_i : \mathbb{K}_i (x_{j_1},\ldots ,x_{j_t}; V_q)(u) \to \mathbb{K}_{i-1} (x_{j_1},\ldots ,x_{j_t}; V_q)(u)
\]
is the alternating sum $d_{i} = \sum_{r=0}^{i-1} (-1)^r d_{i,r}$, where $d_{i,r} \colon \mathbb{K}_i (x_{j_1},\ldots ,x_{j_t}; V_q)(u) \to \mathbb{K}_{i-1} (x_{j_1},\ldots ,x_{j_t}; V_q)(u)$ is the function mapping each summand $H_q(X^{u-e_{\gamma}})$ of $\mathbb{K}_i (x_{j_1},\ldots ,x_{j_t}; V_q)(u)$, with $\gamma = \{ j_{s(1)}, \ldots , j_{s(i)} \}$ and $s(1)<\cdots < s(i)$, to the summand $H_q(X^{u-e_{\gamma}+e_{j_{s(i-r)}}})$ in $\mathbb{K}_{i-1} (x_{j_1},\ldots ,x_{j_t}; V_q)(u)$
via the function $\iota_q^{u-e_{\gamma}, \, u-e_{\gamma}+e_{j_{s(i-r)}}}$.

For any $k \in [n]\smallsetminus J$,
the Koszul complex $\mathbb{K}_* (x_{j_1},\ldots ,x_{j_t}, x_k; V_q)(u)$ is the mapping cone of the chain map induced by inclusions in direction $e_k$,
\[
f^k(x_{j_1},\ldots ,x_{j_t}; V_q)(u-e_k) : \mathbb{K}_* (x_{j_1},\ldots ,x_{j_t}; V_q)(u-e_k) \to \mathbb{K}_* (x_{j_1},\ldots ,x_{j_t}; V_q)(u) ,
\]
which for each degree $i\in \{ 0,1,\ldots , |J|\}$ is defined by 
\[
f_i^k(x_{j_1},\ldots ,x_{j_t}; V_q)(u-e_k) = \bigoplus_{\gamma \subseteq J, \; |\gamma|=i} \iota_q^{u-e_{\gamma}-e_k, \, u -e_{\gamma}} .
\]

In Section \ref{sec:support_xi}, several results will be obtained by showing certain mapping cones to be acyclic, i.e.\ having vanishing homology in all degrees.  
We recall the following immediate consequence of \cite[Prop.~A3.19]{eisenbud1995commutative} (see also \cite[Corollary~1.5.4]{weibel1994introduction}), which gives an equivalent condition to the acyclicity of a mapping cone.

\begin{prop}
\label{prop:coroWeibel}
A chain map $f:B_* \to C_*$ is a quasi-isomorphism (i.e., it induces isomorphisms $H_q(B_*)\cong H_q(C_*)$ in homology, for all $q\in\Z$) if and only if $\cone (f)_*$ is acyclic.
\end{prop}

\begin{coro}
\label{coro:mapping_acyclic}
Let $f:B_* \to C_*$ be a chain map, and let $B_*$ and $C_*$ be acyclic. Then $\cone (f)_*$ is acyclic.
\end{coro}
\begin{proof}
If $B_*$ and $C_*$ are acyclic, the chain map $f$ must be a quasi-isomorphism.
\end{proof}

\section{Entrance grades and support of Betti tables via Koszul complexes}
\label{sec:support_xi}

In this section, we resume the investigation, started in Section \ref{sec:free}, of the relations between the set of entrance grades of cells in a one-critical filtration $\{X^u\}_{u\in\N^n}$, and the set of grades where the Betti tables of the persistent homology module $V_q=\bigoplus_{u\in\N^n}H_q(X^u)$ are nonzero. The main tool of the approach we propose is the Koszul complex.
In Section \ref{subsec:suppallBetti} we prove a result (Theorem \ref{thm:supp_xi_m}) on the support of Betti tables $\xi_i^q$ of all degrees $i\in\{0,\ldots ,n\}$ which improves the results of Section \ref{sec:free}. In Section \ref{subsect:xi1npar}, we focus on the Betti table $\xi_1^q$, stating a sufficient condition for its vanishing at a given grade in terms of the submodules of cycles and boundaries of $V_q$ (Theorem \ref{thm:H1K}). This result can be used to better approximate the support of $\xi_1^q$. In Section \ref{subsec:Morse-supp}, we explain how the presented results can be combined with reductions of the filtered cell complex via discrete Morse theory.

Our fixed setting for the whole section will be as in Section \ref{sec:free}. For the reader's convenience, we briefly recall it. Let $\{ X^u \}_{u\in \N^n}$ be a one-critical (Section \ref{subsec:multipers}) and exhaustive $n$-parameter filtration of a cell complex $X$, which is also graded by the dimension $q$ of cells.
To study the connections with discrete Morse theory (in Section \ref{subsec:Morse-supp}), we consider a fixed discrete gradient vector field $\V$ consistent with the filtration (see Section \ref{subsec:DMT}), and denote by $\{ M^u \}_{u\in \N^n}$ the associated $n$-parameter filtration of the Morse complex $M$. Extending a notation used in Section \ref{sec:free}, we denote set of entrance grades (Section \ref{subsec:multipers}) of a non-empty subset $\Gamma$ of cells of $X$ by
\[
\G (\Gamma) \= \{ \text{entrance grades of the cells of $\Gamma$}\} \subseteq \N^n . 
\]
We denote by $\overline{G}$ the closure of a non-empty subset $G\subseteq \N^n$ with respect to the least upper bound in $\N^n$, which is the set $\overline{G}:=\{\bigvee L \mid L\subseteq G, L\ne \emptyset \}\subseteq \N^n$. Moreover, we recall that $\supp \xi^q_i := \{ u \in \N^n \mid \xi^q_i (u) \ne 0 \}$ denotes the support of the $i$th Betti table $\xi^q_i : \N^n \to \N$ of the persistent homology module $V_q=\bigoplus_{u\in\N^n}H_q(X^u)$. 
Lastly, we establish a notation that will be used throughout this section and state two simple results that will be instrumental in studying the support of the Betti tables using Koszul complexes.
\begin{notation}
\label{not:w}
Having fixed a grade $u\in \N^n$, for any $\alpha \subseteq [n]$ we set $w(\alpha):= u - e_{\alpha}$, where $e_{\alpha}:=\sum_{j\in \alpha} e_j$.
\end{notation}
\begin{lem}\label{lem:isos}
Let $A,B,C,D$ be subspaces of a vector space $V$ over the field $\F$. Suppose that $B\subseteq A\subseteq C$ and $B\subseteq D\subseteq C,$ and let $f:\frac{A}{B}\to \frac{C}{D}$ be the linear map induced by the inclusion of $A$ in $C$. Then there are canonical isomorphisms
\[
\ker f = \frac{A\cap D}{B}, \qquad \im f \cong \frac{A}{A\cap D} \cong  \frac{A+D}{D}, \qquad \coker f \cong \frac{C}{A+D}.
\]
\end{lem}
\begin{proof}
Let $\varphi$ denote the composition $A\hookrightarrow C \twoheadrightarrow \frac{C}{D}$ of the canonical injection and projection. The map $f$, induced by $\varphi$ on the quotient, is well defined since $\ker \varphi = A\cap D \supseteq B$, and satisfies $\ker f = \frac{\ker \varphi}{B}=\frac{A\cap D}{B}$ and $\im f = \im \varphi \cong \frac{A}{\ker \varphi}$, see e.g.\ \cite[p.~19]{Atiyah1969}. The remaining canonical isomorphisms of the claim are obtained via the standard isomorphism theorems \cite[Prop.~2.1]{Atiyah1969}.
\end{proof}
\begin{lem}\label{lem:exact0}
Let $A \xrightarrow{f} B \xrightarrow{g} C \xrightarrow{h} D \xrightarrow{i} E$ be an exact sequence of vector spaces over the fixed field $\F$. Then $C=0$ if, and only if, $f$ is surjective and $i$ is injective. 
\end{lem}
\begin{proof}
If $f$ is surjective and $i$ is injective, then $\ker g = \im f = B$, which implies $\ker h=\im g =0$, and therefore $C\cong C/\ker h \cong \im h = \ker i = 0$.
Conversely, if $C=0$, then $\im f = \ker g =B$ and $\ker i = \im h =0$.
\end{proof}

\subsection{Results on the support of all Betti tables}
\label{subsec:suppallBetti}

Our goal for this subsection is to prove that $\bigcup_{i=0}^n \supp \xi^q_i   \subseteq \overline{\G (X_q)} \cup \overline{\G (X_{q+1})}$ and, moreover, $\supp \xi^q_0   \subseteq \overline{\G (X_q)}$ and $\supp \xi^q_n   \subseteq \overline{\G (X_{q+1})}$, for all $q\in \N$ (Theorem \ref{thm:supp_xi_m}). 
For $\xi_1^q$, the result is improved in Section \ref{subsect:xi1npar}.
We observe that the first inclusion is clearly equivalent to the following statement: if $u\notin \overline{\G (X_q)} \cup \overline{\G (X_{q+1})}$, then $\xi^q_i(u) = 0$, for all $i \in \{ 0,1,\ldots ,n\}$.
To start with, we prove a result that allows us to rephrase the hypothesis of this statement. 

\begin{prop}
\label{prop:equiv_u_lub}
Let $A$ be any subset of cells of $X$ and let $u\in \N^n$. Then $u\notin \overline{\G (A)}$ if and only if there exists $j\in [n]$ such that for any subset $\alpha_j \subseteq [n] \smallsetminus \{ j \}$ it holds $(X^{w(\alpha_j)} \smallsetminus X^{w(\alpha_j)-e_j})\cap A = \emptyset$, where $w(\alpha_j)$ is defined as in Notation \ref{not:w}.
\end{prop}
\begin{proof}
We prove the contrapositive claim, showing the equivalence of the following statements:
\begin{enumerate}
\item $u\in \overline{\G (A)}$.
\item For all $j\in [n]$, there exists a subset $\alpha_j \subseteq  [n] \smallsetminus \{j \}$ such that $\left( X^{w(\alpha_j)} \smallsetminus X^{w(\alpha_j)-e_j} \right) \cap A \ne \emptyset$.
\end{enumerate}

Assume that $u\in \overline{\G (A)}$. If $u\in \G (A)$, we are done by taking $\alpha_j = \emptyset$, for all $j$. If $u\notin \G (A)$, then $u=\bigvee \{ v_1, \ldots ,v_r \}$ with $r\ge 2$ and $v_1,\ldots v_r \in \G (A)$. 
In this case, by definition of the least upper bound, for all $j\in [n]$ there exists $\ell(j) \in [r]$ such that $u-e_j \not\succeq v_{\ell(j)}$. 
Therefore, taking a  cell $\sigma_{\ell(j)}\in A$ with entrance grade $v_{\ell(j)}$, we have $\sigma_{\ell(j)} \in \left( X^{u} \smallsetminus X^{u-e_j} \right) \cap A$, since $u-e_j \not\succeq v_{\ell(j)}$ implies $\sigma_{\ell(j)} \notin X^{u-e_j}$ by one-criticality of the multifiltration (Section \ref{subsec:multipers}). 
The second statement follows again by taking $\alpha_j = \emptyset$, for all $j$.

Conversely, assume that the second statement  holds. For each $j\in [n]$, let $v(j)$ denote the entrance grade of a cell $\sigma_j \in \left( X^{w(\alpha_j)} \smallsetminus X^{w(\alpha_j)-e_j} \right) \cap A$, for some $w(\alpha_j)= u - \sum_{i \in \alpha_j}e_i$. Let $v=\bigvee \{ v(1),\ldots ,v(n) \}$. From $v(j)\preceq w(\alpha_j) \preceq u$, for all $j$, we see that $v\preceq u$. Let us show now that $v=u$, which concludes the proof. If $v\ne u$, then there exists $j\in [n]$ such that $v\preceq u-e_j$. 
Since $\sigma_j$ has entrance grade $v(j)$ and $v(j)\preceq v\preceq u-e_j$, we have  $\sigma_j \in X^{u-e_j}$. On the other hand, we are assuming that $\sigma_j \in X^{w(\alpha_j)}$ with $w(\alpha_j)= u - \sum_{i \in \alpha_j}e_i$ and $j\notin \alpha_j$. The latter condition implies that $w(\alpha_j)$ and $u-e_j$ are not comparable. More precisely, the greatest lower bound of $w(\alpha_j)$ and $u-e_j$ is $w(\alpha_j)-e_j$. Hence, the one-criticality assumption on the multifiltration yields a contradiction (see Remark \ref{rmk:one-criticality}), since we are assuming that $\sigma_j \notin X^{w(\alpha_j)-e_j}$.
\end{proof}

We underline that the one-criticality assumption on the $n$-filtration $\{ X^u \}_{u\in \N^n}$ plays a key role in the proof of Proposition \ref{prop:equiv_u_lub}. 

\begin{coro}
\label{coro:A_Mq}
For any $u\in \N^n$, we have
$u\notin \overline{\G (X_q)}$ if and only if there exists $j\in [n]$ such that
$X_q^{w(\alpha_j)-e_j} = X_q^{w(\alpha_j)}$, for all subsets $\alpha_j \subseteq [n] \smallsetminus \{ j \}$.
\end{coro}

Proposition \ref{prop:equiv_u_lub} also yields information on the maps of the persistent homology modules $\{ H_q(X^u), \iota_q^{u,v} \}$ and $\{ H_{q-1}(X^u), \iota_{q-1}^{u,v} \}$ in the ``vicinity'' of a fixed grade $u\notin \overline{\G (X_q)}$.

\begin{coro}
\label{coro:u_not_lub}
If $u\notin \overline{\G (X_q)}$, then there exists $j\in [n]$ such that, for all 
$\alpha_j \subseteq [n] \smallsetminus \{ j\}$, the inclusion $X^{w(\alpha_j)-e_j}\hookrightarrow X^{w(\alpha_j)}$ induces a surjection
\[
\iota_{q}^{w(\alpha_j)-e_j , w(\alpha_j)} : H_{q}(X^{w(\alpha_j)-e_j}) \to H_{q}(X^{w(\alpha_j)})
\]
and an injection
\[
\iota_{q-1}^{w(\alpha_j)-e_j , w(\alpha_j)} : H_{q-1}(X^{w(\alpha_j)-e_j}) \to H_{q-1}(X^{w(\alpha_j)}) .
\]
\end{coro}
\begin{proof}
By Proposition \ref{prop:equiv_u_lub}, if $u \notin \overline{\G (X_q)}$, then there exists $j\in [n]$ such that,  
for all $\alpha_j \subseteq [n] \smallsetminus \{ j\}$,
we have $X_q^{w(\alpha_j)-e_j} = X_q^{w(\alpha_j)}$, which implies $H_{q}(X^{w(\alpha_j)},X^{w(\alpha_j)-e_j})=0$.
The claim follows from the following portion of the long exact sequence of relative homology of $(X^{w(\alpha_j)},X^{w(\alpha_j)-e_j})$,
\[
H_{q}(X^{w(\alpha_j)-e_j}) \xrightarrow{\quad}
H_{q}(X^{w(\alpha_j)}) \xrightarrow{\quad}
0 \xrightarrow{\quad}
H_{q-1}(X^{w(\alpha_j)-e_j}) \xrightarrow{\quad}
H_{q-1}(X^{w(\alpha_j)}) ,
\]
where the first map is $\iota_{q}^{w(\alpha_j)-e_j, w(\alpha_j)}$ and the last map is $\iota_{q-1}^{w(\alpha_j)-e_j, w(\alpha_j)}$.
\end{proof}

\begin{remark}
\label{rmk:12}
Moving towards the proof of our main result, let us note that the hypothesis $u\notin \overline{\G (X_q)} \cup \overline{\G (X_{q+1})}$ implies, applying Corollary \ref{coro:A_Mq} twice, that the following properties hold simultaneously:
\begin{enumerate}
    \item[(i)] there exists $j\in [n]$ such that $X_q^{w(\alpha_j)-e_j} = X_q^{w(\alpha_j)}$, for all subsets $\alpha_j \subseteq [n] \smallsetminus \{ j \}$.
    \item[(ii)] there exists $\ell \in [n]$ such that  $X_{q+1}^{w(\alpha_{\ell})-e_\ell} = X_{q+1}^{w(\alpha_{\ell})}$, for all subsets $\alpha_{\ell} \subseteq [n] \smallsetminus \{ \ell \}$.
\end{enumerate}
\end{remark}

Clearly, the indices $j$ and $\ell$ of properties (i) and (ii) in Remark \ref{rmk:12} can either coincide or not.
We next prove that both cases imply the acyclicity of certain Koszul complexes, addressing the case $j=\ell$ in Lemma \ref{lem:j_equal_l} and the case $j\ne \ell$ in Lemma \ref{lem:j_notequal_l}. 

\begin{lem}
\label{lem:j_equal_l}
If properties (i) and (ii) in Remark \ref{rmk:12} are verified with  $j=\ell$,  then the Koszul complex $\mathbb{K}_* (x_1 , \ldots , x_n; V_q)(u)$ is acyclic. 
\end{lem}
\begin{proof}
Reasoning as in the proof of Corollary \ref{coro:u_not_lub}, we see that
the maps
\[
\iota_{q}^{w(\alpha_j)-e_j , w(\alpha_j)} : H_{q}(X^{w(\alpha_j)-e_j}) \to H_{q}(X^{w(\alpha_j)})
\]
are isomorphisms, for all subsets $\alpha_j \subseteq [n] \smallsetminus \{ j\}$.
Therefore, the induced chain map
\[
f^j(x_1 , \ldots , \hat{x}_j, \ldots , x_n; V_q)(u-e_j):\mathbb{K}_* (x_1 , \ldots , \hat{x}_j, \ldots , x_n; V_q)(u-e_j) \to \mathbb{K}_* (x_1 , \ldots , \hat{x}_j, \ldots , x_n; V_q)(u)
\] 
is an isomorphism of chain complexes. Hence, the claim follows from Proposition \ref{prop:coroWeibel}
because $\mathbb{K}_* (x_1 , \ldots , x_n; V_q)(u)$ is the mapping cone of $f^j(x_1 , \ldots , \hat{x}_j, \ldots , x_n; V_q)(u-e_j)$.
\end{proof}

\begin{lem}
\label{lem:j_notequal_l}
Let $u\in \N^n$ and suppose that properties (i) and (ii) of Remark \ref{rmk:12} hold with $j\ne \ell$. Then, for any $w\=w(\alpha)= u-e_{\alpha}$ with $\alpha \subseteq [n]\smallsetminus \{ j,\ell\}$, the Koszul complex $\mathbb{K}_* (x_j, x_\ell; V_q)(w)$ is acyclic.
\end{lem}
\begin{proof}
In order to apply Proposition \ref{prop:coroWeibel}, we regard $\mathbb{K}_* (x_j, x_\ell; V_q)(w)$ as the mapping cone of the chain map
\[
f^{\ell}(x_j;V_q)(w-e_{\ell}) : 
\mathbb{K}_* (x_j; V_q)(w-e_{\ell}) \to \mathbb{K}_* (x_j; V_q)(w) .
\]
We want to prove that $f^{\ell}(x_j;V_q)(w-e_{\ell})$ induces isomorphisms between the homology modules of
\[
\mathbb{K}_* (x_j; V_q)(w-e_{\ell}) = \left(0\to H_q(X^{w-e_j -e_\ell}) \xrightarrow{d_1 = \iota_q^{w-e_j -e_\ell, w-e_\ell}} H_q(X^{w-e_\ell}) \to 0 \right)
\]
and
\[
\mathbb{K}_* (x_j; V_q)(w) = \left(0\to H_q(X^{w-e_j}) \xrightarrow{d_1 = \iota_q^{w-e_j,w}} H_q(X^{w}) \to 0 \right) .
\]
Since $\iota_q^{w-e_j -e_\ell, w-e_\ell}$ and $\iota_q^{w-e_j,w}$ are surjective (see proof of Corollary \ref{coro:u_not_lub}), homology in degree $0$ is zero for both Koszul complexes. Hence, we only have to show that
\[
f' : \ker \iota_q^{w-e_j -e_\ell, w-e_\ell} \to \ker \iota_q^{w-e_j,w}
\]
is an isomorphism, where $f'$ denotes the restriction of $\iota_q^{w-e_j -e_\ell, w-e_j}$ to $\ker \iota_q^{w-e_j -e_\ell, w-e_\ell}$. The map $f'$ is injective because  $\iota_q^{w-e_j -e_\ell, w-e_j}$ is injective (see proof of Corollary \ref{coro:u_not_lub}). 
We now show that $f'$ is surjective.
We use here the notations $Z_q(X^v)$ and $B_q(X^v)$ respectively for the submodules of cycles and boundaries of $C_q(X^v)$, for all $v\in \N^n$.
By  Remark~\ref{rmk:12}(i), $X_q^{w-e_j-e_\ell} = X_q^{w-e_\ell}$ and $X_q^{w-e_j} = X_q^{w}$, which implies  $Z_q(X^{w-e_j-e_\ell}) = Z_q(X^{w-e_\ell})$ and $Z_q(X^{w-e_j}) = Z_q(X^{w})$. Similarly, by Remark \ref{rmk:12}(ii), $X_{q+1}^{w-e_j-e_\ell} = X_{q+1}^{w-e_j}$ and $X_{q+1}^{w-e_\ell} = X_{q+1}^{w}$, which implies  $B_{q}(X^{w-e_j-e_\ell}) = B_{q}(X^{w-e_j})$ and $B_{q}(X^{w-e_\ell}) = B_{q}(X^{w})$. 
By Lemma \ref{lem:isos}, we have
\[
\ker \iota_q^{w-e_j -e_\ell, w-e_\ell}=\frac{Z_q(X^{w-e_j-e_\ell})\cap B_q(X^{w-e_\ell})}{B_q(X^{w-e_j-e_\ell})},  \qquad  \ker \iota_q^{w-e_j,w}=\frac{Z_q(X^{w-e_j})\cap B_q(X^{w})}{B_q(X^{w-e_j})}.
\]
Since $f'$ is the map induced by the inclusion of the numerators, using Lemma \ref{lem:isos} and the equalities of subspaces $Z_q$ and $B_q$ stated above, we obtain 
\[
\coker f' \cong \frac{Z_q(X^{w-e_j})\cap B_q(X^{w})}{Z_q(X^{w-e_j-e_\ell})\cap B_q(X^{w-e_\ell}) + B_q(X^{w-e_j})} =\frac{B_q(X^w)}{B_q(X^{w})},
\]
proving that $f'$ is surjective, hence an isomorphism.
\end{proof}

We underline that to conclude the proof we use an argument based on the equality of some subsets of cells of $X$. This part of the proof cannot be replaced by using only the properties of the induced maps in homology (as in Corollary \ref{coro:u_not_lub}). As a counterexample, consider the diagram
\[\begin{tikzcd}
	0 & 0 \\
	V & 0
	\arrow[from=1-1, to=1-2]
	\arrow[from=2-1, to=2-2]
	\arrow[from=1-1, to=2-1]
	\arrow[from=1-2, to=2-2]
\end{tikzcd}\]
of vector spaces, with $\dim V\ne 0$. We can regard the rows as two chain complexes with surjective differentials, and the vertical arrows as an injective chain map between them, as in our proof. However, the mapping cone of this chain map is clearly not acyclic.

We can now complete the proof of our main result for this section.

\begin{thm}
\label{thm:supp_xi_m}
Let $\{ X^u \}_{u\in \N^n}$ be an $n$-parameter exhaustive filtration of a cell complex $X$. Then
\[ \bigcup_{i=0}^n \supp \xi^q_i   \subseteq \overline{\G (X_{q+1})} \cup \overline{\G (X_{q})},
\]
for all $q\in \N$. Furthermore, $\supp \xi^q_0   \subseteq \overline{\G (X_q)}$ and $\supp \xi^q_n   \subseteq \overline{\G (X_{q+1})}$, for all $q\in \N$.
\end{thm}

\begin{proof}
To prove that $\bigcup_{i=0}^n \supp \xi^q_i   \subseteq \overline{\G (X_{q+1})} \cup \overline{\G (X_{q})}$, let $u\notin \overline{\G (X_{q+1})} \cup \overline{\G (X_{q})}$. 
As we have seen, properties (i) and (ii) of Remark \ref{rmk:12} hold, which involve indices $j,\ell\in [n]$. If $j=\ell$, the Koszul complex $\mathbb{K}_* (x_1,\ldots ,x_n; V_q)(u)$ is acyclic by Lemma \ref{lem:j_equal_l}.  
If $j\ne \ell$, consider the Koszul  complexes $\mathbb{K}_* (x_j, x_\ell; V_q)(w)$, for any $w\=w(\alpha)= u-\sum_{i\in \alpha}e_i$ with $\alpha \subseteq [n]\smallsetminus \{ j,\ell\}$, which are acyclic by Lemma \ref{lem:j_notequal_l}.
The Koszul complex  $\mathbb{K}_* (x_1 , \ldots , x_n; V_q)(u)$ can be obtained from the chain complexes $\mathbb{K}_* (x_j, x_\ell; V_q)(w)$ by iterating the mapping cone construction (see Section \ref{sec:Koszul}). At each iteration of this process, by Corollary \ref{coro:mapping_acyclic}, one obtains acyclic Koszul complexes, hence we can conclude that $\mathbb{K}_* (x_1,\ldots ,x_n; V_q)(u)$ is acyclic, that is, $\xi^q_i (u)=0$ for all $i\in \{0,\ldots , n\}$.

To prove that $\supp \xi^q_0 \subseteq \overline{\G (X_q)}$, we observe that if $u\notin \overline{\G (X_q)}$ then by Corollary \ref{coro:u_not_lub} there exists $j\in [n]$ such that $H_q(X^{u-e_j})\to H_q(X^{u})$ is surjective. This implies that the differential $d_1$ of the Koszul complex $\mathbb{K}_* (x_1,\ldots ,x_n; V_q)(u)$ is surjective, hence $\xi_0^{q}(u) = \dim (H_q(X^{u})/\im d_1)=0$. 

Similarly, to prove that $\supp \xi^q_n \subseteq \overline{\G (X_{q+1})}$, we observe that if $u\notin \overline{\G (X_{q+1})}$ then by Corollary~\ref{coro:u_not_lub} there exists $j\in [n]$ such that $H_q(X^{w-e_j})\to H_q(X^{w})$ is injective, where $w:= u- \sum_{i\in [n]\setminus \{j\} }e_i$. This implies that the differential $d_n$ of the Koszul complex $\mathbb{K}_* (x_1,\ldots ,x_n; V_q)(u)$ is injective, hence $\xi_n^{q}(u)= \dim \ker d_n =0$.
\end{proof}

The following simple consequence of Theorem \ref{thm:supp_xi_m} gives a bound of the union of the support of the Betti tables over all the homology degrees inside the union of the closures of the sets of entrance grades of critical cells over all the dimensions. 
\begin{coro}
Under the assumptions of Theorem \ref{thm:supp_xi_m},
$\bigcup_{q,i=0}^n \supp \xi^q_i   \subseteq \bigcup_{q=0}^n\overline{\G (X_q)}$.    
\end{coro}
The following Example \ref{ex:tight} shows that in general the right-hand side term of this inclusion cannot be reduced to a smaller set defined in terms of the entrance grades of cells of $X$, making this inclusion tighter in some sense. A more refined version of it when $n=2$ will be given in the next section (cf.\ Corollary \ref{coro:supp_2par}) in terms of homological critical grades. 

\begin{ex}\label{ex:tight} 
Let $n=3$ and let $X$ be the following simplicial complex:
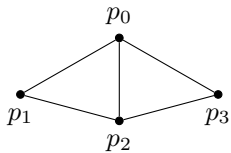
\begin{figure}[!h]
\centering
\begin{tikzpicture}
      \tikzset{enclosed/.style={draw, circle, inner sep=0pt, minimum size=.1cm, fill=black}}
      \node[enclosed, label={below, yshift=0cm: $p_1$}] (v1) at (0,0.35) {};
      \node[enclosed, label={below, xshift=0cm: $p_2$}] (v2) at (1.3,0) {};
      \node[enclosed, label={below, yshift=0cm: $p_3$}] (v3) at (2.6,0.35) {};
      \node[enclosed, label={above, xshift=0cm: $p_0$}] (v0) at (1.3,1.1) {};
      \draw (v0) -- (v1);
      \draw (v0) -- (v2);
      \draw (v0) -- (v3);
      \draw (v1) -- (v2);
      \draw (v2) -- (v3);
\end{tikzpicture}
\end{figure}\\
Let us consider the following $3$-filtration of $X$: all vertices and the edges $p_1 p_2$ and $p_2 p_3$ have entry grade $\mathbf{0}=(0,0,0)\in \N^3$; for all $j\in \{1,2,3\}$, let the edge $p_0 p_j$ have entry grade $u_j \coloneqq \lambda_j e_j$, for some positive integer $\lambda_j$. Figure \ref{fig:counterexZ} in Section \ref{subsect:xi1npar} represents a filtration of this form. Then, all entry grades and all their least upper bounds in $\N^3$ are in $\supp \xi^q_i$ for some $q$ and $i$:  $\xi^0_0(\mathbf{0})=2$,  $\xi^0_1(u_j)=1$ for all $j$, $\xi^1_0(u_j\vee u_k)=1$ for all $j\ne k$, and $\xi^1_1 (u_1 \vee u_2 \vee u_3)=1$. This example can be generalized to any $n\ge 1$. 
\end{ex}

\subsection{A condition for the vanishing of $\xi_1^q$}
\label{subsect:xi1npar}

As in the rest of the section, our starting point is an $n$-parameter exhaustive filtration $\{ X^u \}_{u\in \N^n}$ of a cell complex $X$, of which we consider the $q$th persistent homology module regarded as the $n$-graded $S$-module $V_q = \bigoplus_{u\in \N^n}H_q (X^u)$. For each $u\in\N^n$, we write $H_q (X^u)=\frac{Z_q(X^u)}{B_q(X^u)}$, where $Z_q(X^u)= \ker (\partial_q \colon C_q(X^u)\to C_{q-1}(X^u))$ and $B_q(X^u)= \im (\partial_{q+1} \colon C_{q+1}(X^u)\to C_{q}(X^u))$. We observe that $Z_q \coloneqq \bigoplus_{u\in \N^u}Z_q(X^u)$ and $B_q \coloneqq \bigoplus_{u\in \N^u}B_q(X^u)$ are $n$-graded $S$-modules, respectively given by the kernel of the $n$-graded homomorphism $\partial_q \colon \bigoplus_{u\in\N^n}C_q(X^u)\to \bigoplus_{u\in\N^n}C_{q-1}(X^u)$ and the image of the $n$-graded homomorphism $\partial_{q+1} \colon \bigoplus_{u\in\N^n}C_{q+1}(X^u)\to \bigoplus_{u\in\N^n}C_{q}(X^u)$. In this subsection, we give a condition for the vanishing of the Betti table $\xi_1^q$ of $V_q$ at a fixed $u\in \N^n$ in terms of $B_q(X^v)$ and $Z_q(X^v)$ at grades $v\in \{u-e_\alpha\}_{\alpha \subseteq [n]}$ (Theorem \ref{thm:H1K}), and we derive relations between the support of $\xi_1^q$ and the entrance grades of cells.

Our aim is studying, for any fixed $u\in \N^n$, the degree--1 homology  of the Koszul complex $\K_* (x_1, \ldots ,x_n; V_q)(u)$, whose dimension is the value $\xi_1^q(u)$.

We fix $u\in \N^n$ and $q\in \N$. We choose $\ell \in [n]$ and define an $n$-filtered cell complex $\{\widetilde{X}^v\}_{v\in \N^n}$ such that its $(q+1)$-cells are
\begin{equation*}
\widetilde{X}^v_{q+1} \coloneqq 
\begin{cases}
\bigcup_{j\in [n]\smallsetminus \{\ell\}} X^{v-e_j}_{q+1} & \text{if } v\in \{u, u-e_\ell\},\\
X^{v}_{q+1}  & \text{otherwise},
\end{cases}
\end{equation*}
its lower dimensional cells are $\widetilde{X}^v_{r}\coloneqq X^v_{r}$ for all $r\le q$ and all $v\in \N^n$, and it does not have any cell of dimension higher than $q+1$. The incidence function of $\widetilde{X}$ is induced (by restriction) by the incidence function of $X$. We remark that, since our goal is studying the Koszul complex at $u$, we will only look at the grades $v\in \{u-e_\alpha\}_{\alpha \subseteq [n]}$ of the filtration $\{\widetilde{X}^v\}_{v\in \N^n}$. The $q$th homology of $\{\widetilde{X}^v\}_{v\in \N^n}$ is the $n$-graded $S$-module $\widetilde{V}_{q}$ such that
\begin{equation}\label{eq:tildeV}
\widetilde{V}^v_q \coloneqq 
\begin{cases}
\frac{Z_q(X^v)}{\sum_{j\ne \ell}B_q(X^{v-e_j})} & \text{if } v\in \{u, u-e_\ell\},\\
H_q(X^v) = \frac{Z_q(X^v)}{B_q(X^v)}  & \text{otherwise}.
\end{cases}
\end{equation}
For the sake of a simpler notation, we do not denote the dependence of $\widetilde{V}_{q}$ on the fixed $u\in \N^n$ and the chosen $\ell \in [n]$.

The module $\widetilde{V}_{q}$, which coincides with $V_q$ for all grades except $u$ and $u-e_\ell$, is useful to prove the results of this subsection. 
We observe that the natural $n$-graded homomorphism $\pi\colon\widetilde{V}_q \to V_q$ is surjective, because pointwise it is the linear map
\begin{equation}\label{eq:pi-v}
\pi^v:\frac{Z_q(X^v)}{\sum_{j\ne \ell}B_q(X^{v-e_j})} \twoheadrightarrow \frac{Z_q(X^v)}{B_q(X^v)} 
\end{equation}
for the grades $v\in \{u,u-e_\ell\}$, and it is the identity on $H_q(X^v)$ for all other grades.
We have therefore the following exact sequence of $n$-graded $S$-modules:
\[
\begin{tikzcd}
	0 & {\ker \pi} & {\widetilde{V}_q} & {V_q} & 0.
	\arrow[from=1-1, to=1-2]
	\arrow[from=1-2, to=1-3]
	\arrow["\pi", from=1-3, to=1-4]
	\arrow[from=1-4, to=1-5]
\end{tikzcd}
\]
Since constructing the Koszul complex at $u$ is an exact operation (see Section \ref{subsec:Kosz1}), we obtain the short exact sequence of chain complexes
\[
\begin{tikzcd}
	0 & {\K_*(x_1,\ldots,x_n;\ker \pi)(u)} & {\K_*(x_1,\ldots,x_n;\widetilde{V}_q)(u)} & {\K_*(x_1,\ldots,x_n;V_q)(u)} & 0
	\arrow[from=1-1, to=1-2]
	\arrow[from=1-2, to=1-3]
	\arrow[from=1-3, to=1-4]
	\arrow[from=1-4, to=1-5]
\end{tikzcd}
\]
and the induced long exact sequence in homology
\begin{equation}\label{eq:long-ker}
\begin{tikzcd}
\cdots
  \arrow{r}
&
H_1(\K_*(\widetilde{V}_q)(u))
  \ar{r}
&
H_{1}(\K_*(V_q)(u))
  \arrow[dll, phantom, ""{coordinate, name=Z}]
  \arrow[dll, rounded corners,
    to path={ -- ([xshift=2ex]\tikztostart.east)
    |- (Z)
    -| ([xshift=-2ex]\tikztotarget.west)
    -- (\tikztotarget)}]
\\
H_0(\K_*(\ker \pi)(u))
  \ar{r}
&
H_0(\K_*(\widetilde{V}_q)(u))
  \ar{r}
&
H_{0}(\K_*(V_q)(u))
  \ar{r}
&
0,
\end{tikzcd}
\end{equation}
where we have suppressed the sequence $(x_1,\ldots ,x_n)$ from the notation of the Koszul complexes for brevity.
The following proposition will be useful to prove the main result of this subsection.

\begin{prop}\label{prop:H0ker}
If $B_q(X^u) = \sum_{j=1}^n B_q(X^{u-e_j})$, then
$H_0(\K_*(\ker \pi)(u))=0$.
\end{prop}
\begin{proof}
We recall that the construction of the Koszul complex $\K_*(\ker \pi)(u)$ involves the graded pieces with grades in $\{u-e_{\alpha}\}_{\alpha \subseteq [n]}$ of the $n$-graded $S$-module $\ker \pi$.
By definition of $\widetilde{V}_q$, we have $(\ker \pi)^v=0$ for all $v$, except for $u$ and $u-e_\ell$. 
We consider the linear map
\[
(\ker \pi)^{u-e_\ell}=\frac{B_q(X^{u-e_\ell})}{\sum_{j\in [n] \smallsetminus \{\ell\}} B_q(X^{u-e_j-e_\ell})} \xrightarrow{\phantom{aa} \eta_q^{u-e_\ell, u}\phantom{aa}} (\ker \pi)^u=\frac{B_q(X^{u})}{\sum_{j\in [n] \smallsetminus \{\ell\}} B_q(X^{u-e_j})}
\]
induced by the inclusion $B_q(X^{u-e_\ell})\subseteq B_q(X^{u})$. Regarded as a chain complex (concentrated in homological degrees 1 and 0), the map $\eta_q^{u-e_\ell, u}$ is isomorphic to the Koszul complex $\K_*(\ker \pi)(u)$, hence $\coker \eta_q^{u-e_\ell, u}\cong H_0(\K_*(\ker \pi)(u))$. We conclude the proof  using Lemma \ref{lem:isos} to compute $\coker \eta_q^{u-e_\ell, u}$:
\[
\coker \eta_q^{u-e_\ell, u} \cong \frac{B_q(X^{u})}{B_q (X^{u-e_\ell})+\sum_{j\in [n] \smallsetminus \{\ell\}} B_q(X^{u-e_j})} = \frac{B_q(X^{u})}{\sum_{j=1}^n B_q(X^{u-e_j})}=0,
\]
where the last equality holds by assumption. 
\end{proof}

Our main result of this subsection gives a condition for $H_1(\K_*(V_q)(u))$, and, equivalently, $\xi^q_1(u)$, to vanish.

\begin{thm}\label{thm:H1K}
Let $\{X^v\}_{v\in \N^n}$ be an $n$-parameter exhaustive filtration of a cell complex $X$. Fix $u\in \N^n$, and suppose that $B_q(X^u) = \sum_{j=1}^n B_q(X^{u-e_j})$ and that there exists a permutation $\rho \in \mathrm{Sym}(n)$ such that, for every $\ell \in [n]$,
\[
Z_q(X^{u-e_{\rho(\ell)}})\cap \left( \sum_{j< \ell}Z_q(X^{u-e_{\rho(j)}})\right) = \sum_{j< \ell}Z_q(X^{u-e_{\rho(j)}-e_{\rho(\ell)}}) .
\]
Then $\xi_1^q(u)=0$.
\end{thm}
\begin{proof}
First, we prove the claim supposing that the hypothesis on the cycles is satisfied by the identity permutation $\rho = \id_{[n]}$.
The proof is by induction on the number $n$ of parameters. 

The base case $n=1$ corresponds to the statement $B_q(X^{u})=B_q(X^{u-1})$ implies $\xi_1^q(u)=0$ (as the condition on $Z_q$ is trivially satisfied), which is true because $\xi_1^q(u)$ is the dimension of the vector space $\ker \iota_q^{u-1,u}:H_q(X^{u-1})\to H_q(X^{u})$ which, by Lemma \ref{lem:isos}, is isomorphic to $(Z_q(X^{u-1}) \cap B_q(X^u))/B_q(X^{u-1})$.  

We now prove the claim for $n$ parameters, supposing it holds for $n-1$ parameters.
For any chosen $\ell \in [n]$, we can consider the module $\widetilde{V}_q$ associated with $\{\widetilde{X}^v\}$, defined as in Equation \eqref{eq:tildeV}. Here, we choose $\ell = n$. Under the hypothesis that $B_q(X^u) = \sum_{j=1}^n B_q(X^{u-e_j})$, in the long exact sequence (\ref{eq:long-ker}) we have $H_0(\K_*(\ker \pi)(u))=0$ by Proposition \ref{prop:H0ker}, so it is sufficient to prove that $H_1(\K_*(\widetilde{V}_q)(u))=0$ to conclude that $H_1(\K_*(V_q)(u))=0$. We highlight that using $\widetilde{V}_q$ in the rest of the proof is convenient, as it is constructed in a way that allows us to use the inductive assumption. 
We write the Koszul complex $\K_*(\widetilde{V}_q)(u) = \K_*(x_1,\ldots ,x_n;\widetilde{V}_q)(u)$ as the mapping cone (see Section~\ref{sec:mappingcones}) of the chain map induced by inclusions in direction $e_n$,
\[
f^n(x_{1},\ldots ,x_{n-1}; \widetilde{V}_q)(u-e_n) : \mathbb{K}_* (x_{1},\ldots ,x_{n-1}; \widetilde{V}_q)(u-e_n) \to \mathbb{K}_* (x_{1},\ldots ,x_{n-1}; \widetilde{V}_q)(u) .
\]
In the rest of this proof, for simplicity we denote this chain map and the two Koszul complexes by $f^n : \mathbb{K}^n_* (\widetilde{V}_q)(u-e_n) \to \mathbb{K}^n_* (\widetilde{V}_q)(u)$. 
We consider the long exact sequence of the mapping cone (see e.g.\ \cite[\S~1.5.2]{weibel1994introduction}) for $\cone (f^n )_* =\K_* (\widetilde{V}_q)(u)$:
\[
\begin{tikzcd}
\cdots
  \arrow{r}
&
H_1(\mathbb{K}_*^{n}(\widetilde{V}_q)(u))
  \ar{r}
&
H_{1}(\K_* (\widetilde{V}_q)(u))
  \arrow[dll, phantom, ""{coordinate, name=Z}]
  \arrow[dll, rounded corners,
    to path={ -- ([xshift=2ex]\tikztostart.east)
    |- (Z)
    -| ([xshift=-2ex]\tikztotarget.west)
    -- (\tikztotarget)}]
\\
H_0(\mathbb{K}_*^n (\widetilde{V}_q)(u-e_n))
  \ar{r}
&
H_0(\mathbb{K}_*^n (\widetilde{V}_q)(u))
  \ar{r}
&
H_{0}(\K_* (\widetilde{V}_q)(u))
  \ar{r}
&
0.
\end{tikzcd}
\]
The Koszul complex $\mathbb{K}_*^n (\widetilde{V}_q)(u)=\mathbb{K}_* (x_{1},\ldots ,x_{n-1}; \widetilde{V}_q)(u)$ is defined from the $(n-1)$-parameter filtration $\{ \widetilde{X}^{u-e_{\alpha}}\}_{\alpha \subseteq [n-1]}$, which allows us to apply the inductive assumption, since $B_q(\widetilde{X}^u)= \sum_{j=1}^{n-1} B_q(\widetilde{X}^{u-e_j})$ and the condition involving the subspaces $Z_q$ is satisfied by $\rho=\id_{[n]}$ restricted to $[n-1]$. Therefore, by the inductive assumption, $H_1(\mathbb{K}_*^n (\widetilde{V}_q)(u))=0$. 
Thus, by Lemma \ref{lem:exact0}, the vanishing of $H_1(\K_*(\widetilde{V}_q)(u))$ is ensured by the injectivity of the function $H_0(f^n)\colon H_0(\mathbb{K}_*^n (\widetilde{V}_q)(u-e_n)) \to H_0(\mathbb{K}_*^n (\widetilde{V}_q)(u))$ in the long exact sequence, which is what we show to hold in the next step of the proof.

We begin by observing that $H_0(\mathbb{K}_*^n (\widetilde{V}_q)(u))$ can be written as follows:
\begin{align*}
H_0(\mathbb{K}_*^n (\widetilde{V}_q)(u)) &\cong \coker \left( \bigoplus_{j=1}^{n-1} H_q(\widetilde{X}^{u-e_j}) \xrightarrow{[ \iota_q^{u-e_{1}, u} \; \cdots \; \iota_q^{u-e_{n-1}, u}]}  H_q(\widetilde{X}^{u})  \right) \\
&= \coker \left( \bigoplus_{j=1}^{n-1} \frac{Z_q(\widetilde{X}^{u-e_j})}{B_q(\widetilde{X}^{u-e_j})} \xrightarrow{[ \iota_q^{u-e_{1}, u} \; \cdots \; \iota_q^{u-e_{n-1}, u}]}  \frac{Z_q(\widetilde{X}^{u})}{\sum_{j=1}^{n-1} B_q(\widetilde{X}^{u-e_j})} \right) \\
&\cong \frac{Z_q(\widetilde{X}^{u})}{\sum_{j=1}^{n-1} Z_q(\widetilde{X}^{u-e_j})} \\
&= \frac{Z_q(X^{u})}{\sum_{j=1}^{n-1} Z_q(X^{u-e_j})}.
\end{align*}
Similarly, there is a canonical isomorphism
\begin{align*}
H_0(\mathbb{K}_*^n (\widetilde{V}_q)(u-e_n)) &\cong \frac{Z_q(X^{u-e_n})}{\sum_{j=1}^{n-1} Z_q(X^{u-e_j-e_n})}.
\end{align*}
Using Lemma \ref{lem:isos}, we see that the kernel of the map $H_0(f^n)\colon H_0(\mathbb{K}_*^n (\widetilde{V}_q)(u-e_n)) \to H_0(\mathbb{K}_*^n (\widetilde{V}_q)(u))$, induced by the inclusion $Z_q(X^{u-e_n})\subseteq Z_q(X^{u})$, is
\[
\ker (H_0(f^n)) \cong \frac{Z_q(X^{u-e_n})\cap (\sum_{j=1}^{n-1}Z_q(X^{u-e_j}))}{\sum_{j=1}^{n-1} Z_q(X^{u-e_j-e_n})} ,
\]
which is zero since the numerator coincides with the denominator by hypothesis. This concludes the proof under the assumption $\rho=\id_{[n]}$.

Lastly, we explain how the proof in the special case of the identity permutation implies the claim for a generic permutation $\rho$ on the set $[n]$. Let $\rho$ be a permutation for which the hypothesis of the theorem is satisfied. Then, we can consider the filtration $\{L^{u-e_{\alpha}}\}_{\alpha \subseteq [n]}$ defined by $L^{u-e_{\alpha}}\coloneqq X^{u-e_{\rho (\alpha)}}$, which satisfies $B_q(L^u) = \sum_{j=1}^n B_q(L^{u-e_j})$ and, for every $\ell \in [n]$, $Z_q(L^{u-e_{\ell}})\cap \left( \sum_{j< \ell}Z_q(L^{u-e_{j}})\right) = \sum_{j< \ell}Z_q(L^{u-e_{j}-e_{\ell}})$. The Koszul complex of the associated $q$th persistent homology module at $u$ is obtained from $\K_*(V_q)(u)=\K_*(x_1,\ldots ,x_n;V_q)(u)$ by permuting the indeterminates, and is therefore isomorphic to it (see Section \ref{sec:mappingcones}). We can therefore apply the proof for the case of the identity permutation to $\{L^{u-e_{\alpha}}\}_{\alpha \subseteq [n]}$ and conclude that $\xi_1^q(u)=0$.
\end{proof}

\begin{rmk}\label{rmk:Zq2param}
The condition on the subspaces $Z_q$ in Theorem \ref{thm:H1K} amounts to $n$ different identities of subspaces of $Z_q(X^u)$. In the proof of the theorem we observed that, when the sum on the left-hand side has zero summands, the corresponding inequality is trivially satisfied. It is worth noticing that the equality corresponding to a sum on the left-hand side with exactly one summand is always satisfied too. In other words, for any pair of distinct indices $j,k\in [n]$, the identity $Z_q(X^{u-e_j})\cap Z_q(X^{u-e_k}) = Z_q(X^{u-e_j-e_k})$ holds true. To see this, we recall that $C_q(X^{u-e_j})\cap C_q(X^{u-e_k}) = C_q(X^{u-e_j-e_k})$ holds by one-criticality as a consequence of $X^{u-e_j}\cap X^{u-e_k} = X^{u-e_j-e_k}$ (see Remark \ref{rmk:one-criticality}), and we observe that $Z_q(X^v)=C_q(X^v) \cap \partial_q^{-1} (0)$ for all $v\in \N^n$, where $\partial_q$ denotes the differential $\partial_q \colon C_q(\cup_v X^v)\to C_{q-1}(\cup_v X^v)$. In particular, for $2$-parameter filtrations, the condition of Theorem \ref{thm:H1K} on the subspaces $Z_q$ always holds. In Figure \ref{fig:counterexZ}, we show a $3$-parameter filtration not satisfying the hypothesis of Theorem \ref{thm:H1K} on the subspaces $Z_q$.
\end{rmk}

\begin{figure}
\centering
\begin{tikzpicture}
      \tikzset{enclosed/.style={draw, circle, inner sep=0pt, minimum size=.08cm, fill=black}}

      \draw[line width = 1.5pt, lightgray] (0, 0) rectangle (2.4, 1.4);
      \draw[line width = 1.5pt, lightgray] (4.4, 0) rectangle (2.4 + 4.4, 1.4);
      \draw[line width = 1.5pt, lightgray] (2.2, 2.2) rectangle (2.4 + 2.2, 1.4 + 2.2);
      \draw[line width = 1.5pt, lightgray] (2.2 + 4.4, 2.2) rectangle (2.4 + 2.2 + 4.4, 1.4 + 2.2);

      \draw[line width = 1.5pt, lightgray] (0, 4.8) rectangle (2.4, 1.4 + 4.8);
      \draw[line width = 1.5pt, lightgray] (4.4, 4.8) rectangle (2.4 + 4.4, 1.4 + 4.8);
      \draw[line width = 1.5pt, lightgray] (2.2, 2.2 + 4.8) rectangle (2.4 + 2.2, 1.4 + 2.2 + 4.8);
      \draw[line width = 1.5pt, lightgray] (2.2 + 4.4, 2.2 + 4.8) rectangle (2.4 + 2.2 + 4.4, 1.4 + 2.2 +4.8);
      
      \draw[-Triangle, line width = 1.5pt, lightgray] (2.6, 0.7) to (4.2, 0.7);
      \draw[-Triangle, line width = 1.5pt, lightgray] (2.6, 0.7 + 4.8) to (4.2, 0.7 + 4.8);
      \draw[-Triangle, line width = 1.5pt, lightgray] (1.2, 1.6) to (1.2, 4.6);
      \draw[-Triangle, line width = 1.5pt, lightgray] (1.2 + 4.4, 1.6) to (1.2 + 4.4, 4.6);
      \draw[-Triangle, line width = 1.5pt, lightgray] (2.6 + 2.2, 0.7 + 4.8 + 2.2) to (4.2 + 2.2, 0.7 + 4.8 + 2.2);
      \draw[-Triangle, line width = 1.5pt, lightgray] (1.2 + 4.4 + 2.2, 1.6 + 2.2) to (1.2 + 4.4 + 2.2, 4.6 + 2.2);
      \draw[-Triangle, line width = 1.5pt, lightgray] (1.6, 1.6) to (1.6 + 0.5, 1.6 + 0.5);
      \draw[-Triangle, line width = 1.5pt, lightgray] (1.6 + 4.4, 1.6) to (1.6 + 0.5 + 4.4, 1.6 + 0.5);
      \draw[-Triangle, line width = 1.5pt, lightgray] (1.6, 1.6 + 4.8) to (1.6 + 0.5, 1.6 + 0.5 + 4.8);
      \draw[-Triangle, line width = 1.5pt, lightgray] (1.6 + 4.4, 1.6 + 4.8) to (1.6 + 0.5 + 4.4, 1.6 + 0.5 + 4.8);
      \draw[line width = 1.5pt, lightgray] (2.6 + 2.2, 0.7 + 2.2) to (2.6 + 0.6 + 2.2, 0.7 + 2.2);
      \draw[-Triangle, line width = 1.5pt, lightgray] (2.6 + 1 + 2.2, 0.7 + 2.2) to (4.2 + 2.2, 0.7 + 2.2);
      \draw[line width = 1.5pt, lightgray] (1.2 + 2.2, 1.6 + 2.2) to (1.2 + 2.2, 1.6 + 1.5 + 2.2);
      \draw[-Triangle, line width = 1.5pt, lightgray] (1.2 + 2.2, 1.6 + 1.9 + 2.2) to (1.2 + 2.2, 4.6 + 2.2);

      \node[enclosed] at (1.2, 0.3) {};
      \node[enclosed] at (1.2, 1.1) {};
      \node[enclosed] at (0.4, 0.6) {};
      \node[enclosed] at (2, 0.6) {};
      \draw (1.2, 0.3) -- (0.4, 0.6); 
      \draw (1.2, 0.3) -- (2, 0.6); 

      \node[enclosed] at (1.2 + 4.4, 0.3) {};
      \node[enclosed] at (1.2 + 4.4, 1.1) {};
      \node[enclosed] at (0.4 + 4.4, 0.6) {};
      \node[enclosed] at (2 + 4.4, 0.6) {};
      \draw (1.2 + 4.4, 0.3) -- (0.4 + 4.4, 0.6); 
      \draw (1.2 + 4.4, 0.3) -- (2 + 4.4, 0.6); 
      \draw (1.2 + 4.4, 1.1) -- (0.4 + 4.4, 0.6); 

      \node[enclosed] at (1.2 + 2.2, 0.3 + 2.2) {};
      \node[enclosed] at (1.2 + 2.2, 1.1 + 2.2) {};
      \node[enclosed] at (0.4 + 2.2, 0.6 + 2.2) {};
      \node[enclosed] at (2 + 2.2, 0.6 + 2.2) {};
      \draw (1.2 + 2.2, 0.3 + 2.2) -- (0.4 + 2.2, 0.6 + 2.2); 
      \draw (1.2 + 2.2, 0.3 + 2.2) -- (2 + 2.2, 0.6 + 2.2); 
      \draw (1.2 + 2.2, 0.3 + 2.2) -- (1.2 + 2.2, 1.1 + 2.2); 

      \node[enclosed] at (1.2 + 4.4 + 2.2, 0.3 + 2.2) {};
      \node[enclosed] at (1.2 + 4.4 + 2.2, 1.1 + 2.2) {};
      \node[enclosed] at (0.4 + 4.4 + 2.2, 0.6 + 2.2) {};
      \node[enclosed] at (2 + 4.4 + 2.2, 0.6 + 2.2) {};
      \draw (1.2 + 4.4 + 2.2, 0.3 + 2.2) -- (0.4 + 4.4 + 2.2, 0.6 + 2.2); 
      \draw (1.2 + 4.4 + 2.2, 0.3 + 2.2) -- (2 + 4.4 + 2.2, 0.6 + 2.2); 
      \draw (1.2 + 4.4 + 2.2, 0.3 + 2.2) -- (1.2 + 4.4 + 2.2, 1.1 + 2.2); 
      \draw (1.2 + 4.4 + 2.2, 1.1 + 2.2) -- (0.4 + 4.4 + 2.2, 0.6 + 2.2); 
      \node[label={above, xshift=0cm: {\footnotesize $u-e_3$}}] at (1.95 + 4.4 + 2.2, 0.85 + 2.2) {};

      \node[enclosed] at (1.2, 0.3 + 4.8) {};
      \node[enclosed] at (1.2, 1.1 + 4.8) {};
      \node[enclosed] at (0.4, 0.6 + 4.8) {};
      \node[enclosed] at (2, 0.6 + 4.8) {};
      \draw (1.2, 0.3 + 4.8) -- (0.4, 0.6 + 4.8); 
      \draw (1.2, 0.3 + 4.8) -- (2, 0.6 + 4.8); 
      \draw (1.2, 1.1 + 4.8) -- (2, 0.6 + 4.8); 

      \node[enclosed] at (1.2 + 4.4, 0.3 + 4.8) {};
      \node[enclosed] at (1.2 + 4.4, 1.1 + 4.8) {};
      \node[enclosed] at (0.4 + 4.4, 0.6 + 4.8) {};
      \node[enclosed] at (2 + 4.4, 0.6 + 4.8) {};
      \draw (1.2 + 4.4, 0.3 + 4.8) -- (0.4 + 4.4, 0.6 + 4.8); 
      \draw (1.2 + 4.4, 0.3 + 4.8) -- (2 + 4.4, 0.6 + 4.8); 
      \draw (1.2 + 4.4, 1.1 + 4.8) -- (0.4 + 4.4, 0.6 + 4.8); 
      \draw (1.2 + 4.4, 1.1 + 4.8) -- (2 + 4.4, 0.6 + 4.8); 
      \node[label={above, xshift=0cm: {\footnotesize $u-e_2$}}] at (1.95 + 4.4, 0.85 + 4.8) {};

      \node[enclosed] at (1.2 + 2.2, 0.3 + 2.2 + 4.8) {};
      \node[enclosed] at (1.2 + 2.2, 1.1 + 2.2 + 4.8) {};
      \node[enclosed] at (0.4 + 2.2, 0.6 + 2.2 + 4.8) {};
      \node[enclosed] at (2 + 2.2, 0.6 + 2.2 + 4.8) {};
      \draw (1.2 + 2.2, 0.3 + 2.2 + 4.8) -- (0.4 + 2.2, 0.6 + 2.2 + 4.8); 
      \draw (1.2 + 2.2, 0.3 + 2.2 + 4.8) -- (2 + 2.2, 0.6 + 2.2 + 4.8); 
      \draw (1.2 + 2.2, 0.3 + 2.2 + 4.8) -- (1.2 + 2.2, 1.1 + 2.2 + 4.8); 
      \draw (1.2 + 2.2, 1.1 + 2.2 + 4.8) -- (2 + 2.2, 0.6 + 2.2 + 4.8); 
      \node[label={above, xshift=0cm: {\footnotesize $u-e_1$}}] at (1.95 + 2.2, 0.85 + 2.2 + 4.8) {};

      \node[enclosed] at (1.2 + 4.4 + 2.2, 0.3 + 2.2 + 4.8) {};
      \node[enclosed] at (1.2 + 4.4 + 2.2, 1.1 + 2.2 + 4.8) {};
      \node[enclosed] at (0.4 + 4.4 + 2.2, 0.6 + 2.2 + 4.8) {};
      \node[enclosed] at (2 + 4.4 + 2.2, 0.6 + 2.2 + 4.8) {};
      \draw (1.2 + 4.4 + 2.2, 0.3 + 2.2 + 4.8) -- (0.4 + 4.4 + 2.2, 0.6 + 2.2 + 4.8); 
      \draw (1.2 + 4.4 + 2.2, 0.3 + 2.2 + 4.8) -- (2 + 4.4 + 2.2, 0.6 + 2.2 + 4.8); 
      \draw (1.2 + 4.4 + 2.2, 0.3 + 2.2 + 4.8) -- (1.2 + 4.4 + 2.2, 1.1 + 2.2 + 4.8); 
      \draw (1.2 + 4.4 + 2.2, 1.1 + 2.2 + 4.8) -- (0.4 + 4.4 + 2.2, 0.6 + 2.2 + 4.8); 
      \draw (1.2 + 4.4 + 2.2, 1.1 + 2.2 + 4.8) -- (2 + 4.4 + 2.2, 0.6 + 2.2 + 4.8); 
      \node[label={above, xshift=0cm: {\footnotesize $u$}}] at (1.95 + 4.4 + 2.2 + 0.2, 0.85 + 2.2 + 4.8) {};
\end{tikzpicture}
\caption{A $3$-parameter filtration $\{X^{u-e_{\alpha}}\}_{\alpha \subseteq \{1,2,3\}}$ of simplicial complexes such that, for $q=1$, the equality $Z_q(X^{u-e_\ell})\cap (Z_q(X^{u-e_j})+Z_q(X^{u-e_k}))=Z_q(X^{u-e_j-e_\ell})+Z_q(X^{u-e_k-e_\ell})$ (see hypothesis of Theorem \ref{thm:H1K}) does not hold, for any choice of different $j,k,\ell$ in $\{1,2,3\}$. 
Using the Koszul complex $\K_* (V_1)(u)$ it is easy to see, by a dimension argument, that $H_1(\K_* (V_1)(u))\cong \F$ and, equivalently, $\xi_1^q(u) =1$.}
\label{fig:counterexZ}
\end{figure}
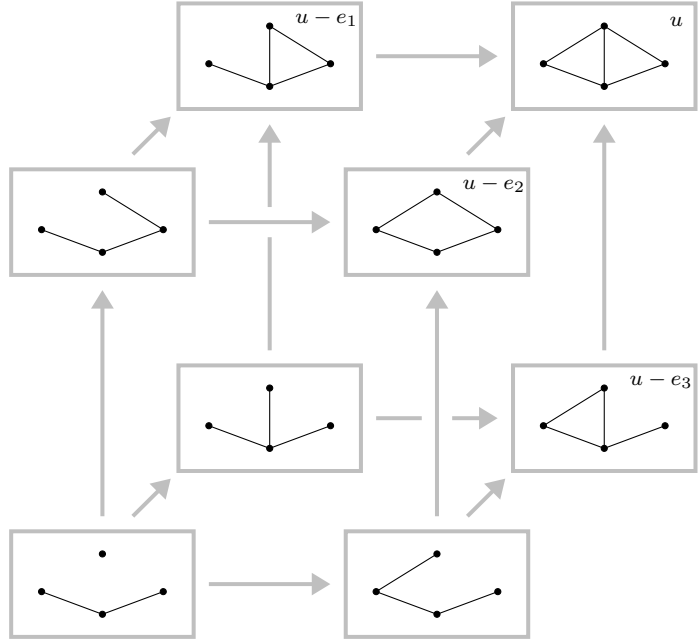

We now state conditions involving the fixed grade $u\in \N^n$ and the sets $\G (X_{q+1})$ and $\G (X_q)$ of entrance grades which ensure that the assumptions in Theorem \ref{thm:H1K} on the subspaces $B_q$ and $Z_q$ are verified.

\begin{prop}\label{prop:suffBq}
If $u\notin \G(X_{q+1})$, then $B_q(X^u)=\sum_{j=1}^n B_q(X^{u-e_j})$.
\end{prop}

\begin{proof}
The inclusion $\sum_{j=1}^n B_q(X^{u-e_j}) \subseteq B_q(X^u)$ holds in general and follows from $B_q(X^{u-e_j}) \subseteq B_q(X^u)$ for every $j\in [n]$. To see the other inclusion, we first observe that $u\notin \G(X_{q+1})$ implies $C_{q+1}(X^u)=C_{q+1}(\cup_{j=1}^n X^{u-e_j})$, which is equal to $\sum_{j=1}^n C_{q+1}(X^{u-e_j})$. The claim then follows from the equalities $\partial_{q+1} (C_{q+1}(X^u))=B_{q}(X^u)$ and $\partial_{q+1} (\sum_{j=1}^n C_{q+1}(X^{u-e_j}))= \sum_{j=1}^n B_{q}(X^{u-e_j})$. 
\end{proof}

\begin{remark}
The converse implication is false as can be seen, for $q=0$, considering a cell complex with two vertices connected by two edges and the following $1$-parameter filtration: the two vertices and one edge enter at grade $u-e_1$, and the other edge enters at grade $u$. 
\end{remark}

\begin{prop}\label{prop:suffZq}
If $u\notin \overline{\G(X_{q})}$, then there exists a permutation $\rho \in \mathrm{Sym}(n)$ such that, for every $\ell \in [n]$,
\[
Z_q(X^{u-e_{\rho(\ell)}})\cap \left( \sum_{j< \ell}Z_q(X^{u-e_{\rho(j)}})\right) = \sum_{j< \ell}Z_q(X^{u-e_{\rho(j)}-e_{\rho(\ell)}}) .
\]    
\end{prop}
\begin{proof}
We prove the statement by induction on the number $n$ of parameters. By Remark \ref{rmk:Zq2param}, for $n=1$ and $n=2$ the identities involving subspaces $Z_q$ hold in general. 

To prove the induction step for $n$ parameters, we recall that by Corollary~\ref{coro:A_Mq} we have $u\notin \overline{\G (X_q)}$ if and only if there exists $k \in [n]$ such that
$X_q^{w(\alpha_k)-e_k} = X_q^{w(\alpha_k)}$, for all subsets $\alpha_k \subseteq [n] \smallsetminus \{ k \}$. We take such an index $k$ and set $\rho(n)\coloneqq k$. 
For any $j\ne k$, taking $\alpha_k=\{j\}$, we get $Z_q(X^{u-e_j-e_k})=Z_q(X^{u-e_j})$. 
Hence, $\sum_{j\in [n]\smallsetminus \{k \}} Z_q(X^{u-e_j-e_k})=\sum_{j\in [n]\smallsetminus \{k \}} Z_q(X^{u-e_j})$, which implies  $ Z_q(X^{u-e_k})\cap \left( \sum_{j\in [n]\smallsetminus \{k\}} Z_q(X^{u-e_j})\right) \subseteq \sum_{j\in [n]\smallsetminus \{k\}} Z_q(X^{u-e_j-e_k})$.
The right-hand side is actually equal to the left-hand side, since the reverse inclusion holds in general and follows from the fact that, for every $j\ne k$, 
\[
Z_q(X^{u-e_j-e_k}) = Z_q(X^{u-e_k})\cap Z_q(X^{u-e_j}) \subseteq Z_q(X^{u-e_k})\cap \left( \textstyle \sum_{j\in [n]\smallsetminus \{k\}} Z_q(X^{u-e_j}) \right),
\]
where the first equality is by Remark \ref{rmk:Zq2param}. This proves the  equality involving subspaces $Z_q$ for $\ell=n$.

Lastly, we have to show that for every $\ell <n$ the remaining equalities involving subspaces $Z_q$ in the claim hold. This is a consequence of the inductive hypothesis, observing that the remaining equalities involve the grades in $\{u-e_{\alpha}\}_{\alpha \subseteq [n]\smallsetminus \{k\}}$, which is a portion of an $(n-1)$-parameter filtration, and that $u$ is not a least upper bound of grades in $\G(X_q)$ belonging to this filtration. By relabeling the parameters in $[n]\smallsetminus \{k\}$ of the $(n-1)$-parameter filtration with indices in $[n-1]$ and applying the inductive assumption, we see that there exists a bijection $\rho': [n-1] \to [n]\smallsetminus \{k\}$ such that the first $n-1$ equalities of the claim hold. We complete the proof by defining $\rho (j)\coloneqq \rho'(j)$, for all $j<n$.
\end{proof}

\begin{rmk}
The converse implication is false in general: for example, even if the equalities involving subspaces $Z_q$ are satisfied, $u$ can be the entrance grade of a $q$-cell of $M$ that does not appear in any $q$-cycle. 
\end{rmk}

Using Proposition \ref{prop:suffBq} and Proposition \ref{prop:suffZq} we immediately obtain the following corollary of Theorem \ref{thm:H1K}. We note that the same bound for the support of $\xi_1^q$ was obtained in Section \ref{sec:free} using multigraded resolutions (Remark \ref{rmk:free-bounds}).

\begin{coro}\label{coro:H1K}
Let $\{X^v\}_{v\in \N^n}$ be an $n$-parameter exhaustive filtration of a cell complex $X$. If $u\in \N^n$ is such that $u\notin \G(X_{q+1})\cup \overline{\G(X_q)}$, then $\xi_1^q(u)=0$. In other words, $\supp \xi_1^q \subseteq \G(X_{q+1})\cup \overline{\G(X_q)}$.
\end{coro}

We end this subsection describing two particular cases in which the equalities involving subspaces $Z_q$ in Theorem \ref{thm:H1K} are always satisfied. The first case corresponds to $q=0$. 

\begin{coro}\label{coro:H1Kq0}
Let $\{X^v\}_{v\in \N^n}$ be an $n$-parameter exhaustive filtration of a cell complex $X$, of which we consider the associated $q$-th persistent homology module with $q=0$. If $u\in \N^n$ is such that $B_0(X^u)=\sum_{j=1}^n B_0 (X^{u-e_j})$, then $\xi_1^0(u)=0$. As a consequence, for $q=0$ we have the following containments: 
\[
\supp \xi_0^0 \subseteq \overline{\G(X_{0})}, \qquad
\supp \xi_1^0 \subseteq \G(X_{1}), \qquad
\bigcup_{i=1}^n \supp \xi_i^0 \subseteq \overline{\G(X_{1})}. 
\]
\end{coro}
\begin{proof}
Since the $n$-parameter filtration $\{X^v\}_{v\in \N^n}$ is one-critical (see Remark \ref{rmk:one-criticality}), the following equalities of (graded) sets hold for all $w,v_1,\ldots ,v_k\in \N^n$:
\[
X^w \cap \left( \textstyle \bigcup_{j=1}^k X^{v_j} \right) = \textstyle \bigcup_{j=1}^k  \left(  X^w \cap X^{v_j} \right) = \textstyle \bigcup_{j=1}^k  X^{w\wedge v_j}.
\]
Considering cells of dimension $q$ and taking the $\F$-linear span of the left-hand and right-hand sides one obtains
\[
C_q(X^w) \cap \left( \textstyle \sum_{j=1}^k C_q(X^{v_j}) \right) =  \textstyle \sum_{j=1}^k  C_q(X^{w\wedge v_j}).
\]
For $q=0$, since $Z_0(X^v)=C_0(X^v)$ for all $v\in \N^n$, all the equalities involving subspaces $Z_q$ in Theorem \ref{thm:H1K} are therefore always satisfied. This proves the first part of the claim and, as an immediate consequence using Proposition \ref{prop:suffBq}, the containment $\supp \xi_1^0 \subseteq \G(X_{1})$. The containment $\bigcup_{i=1}^n \supp \xi_i^0 \subseteq \overline{\G(X_{1})}$ follows from the fact that, for any $n$-graded $S$-module $V$, 
$\bigcup_{i=1}^n \supp \xi_i(V) \subseteq \overline{\supp \xi_1 (V)}$, see for example \cite[Remark~3.2]{charalambous2015betti}.
Lastly, the containment $\supp \xi_0^0 \subseteq \overline{\G(X_{0})}$ holds by Theorem \ref{thm:supp_xi_m}.
\end{proof}

The second particular case corresponds to $n=2$, i.e.\ to bifiltrations of persistent homology modules.

\begin{coro}\label{coro:H1Kn2}
Let $\{X^v\}_{v\in \N^2}$ be an $2$-parameter exhaustive filtration of a cell complex $X$. If $u\in \N^2$ is such that $B_q(X^u)=\sum_{j=1}^n B_q (X^{u-e_j})$, then $\xi_1^q(u)=0$. The following containments hold: 
\[
\supp \xi_0^q \subseteq \overline{\G(X_{q})}, \qquad
\supp \xi_1^q \subseteq \G(X_{q+1}), \qquad
\supp \xi_2^q \subseteq \overline{\G(X_{q+1})}. 
\]
\end{coro}
\begin{proof}
The first part of the claim follows from Theorem \ref{thm:H1K} and Remark \ref{rmk:Zq2param}, and it implies $\supp \xi_1^q \subseteq \G(X_{q+1})$ by Proposition \ref{prop:suffBq}. The other containments hold by Theorem \ref{thm:supp_xi_m}.
\end{proof}

\subsection{Morse complex and support of the Betti tables}
\label{subsec:Morse-supp}

We conclude the section by observing how the results  of Section \ref{subsec:suppallBetti} and Section \ref{subsect:xi1npar} can be applied to the Morse complex $M$ associated with any discrete gradient vector field $\V$ consistent with the filtration $\{X^u\}_{u\in\N^n}$ of $X$ (see Section \ref{subsec:DMT}). By Proposition \ref{prop:iso_pers_mod}, the persistent homology module $V'_q\coloneqq \bigoplus_{u\in\N^n} H_q(M^u)$ associated with $\{M^u\}_{u\in\N^n}$ is isomorphic to $V_q$, hence the Betti tables of $V'_q$ coincide with the Betti tables $\xi_i^q$ of $V_q$. This can be seen for example from the fact that, as observed at the end of Section \ref{subsec:Kosz1}, the Koszul complexes of $V'_q$ and $V_q$ at any $u\in\N^n$ are isomorphic. Therefore, one can bound the support of the Betti tables of $V_q$ using the entrance grades of the cells of $M$.
For example, Theorem \ref{thm:supp_xi_m} has the following immediate consequence.

\begin{coro}
\label{coro:supp_xi_m}
Let $\{ X^u \}_{u\in \N^n}$ be an $n$-parameter exhaustive filtration of a cell complex $X$,
let $\V$ be a fixed discrete gradient vector field consistent with the filtration, and let $\{ M^u \}_{u\in \N^n}$ be the associated $n$-parameter filtration of the Morse complex $M$. Then
\[ \bigcup_{i=0}^n \supp \xi^q_i   \subseteq \overline{\G (M_{q+1})} \cup \overline{\G (M_{q})},
\]
for all $q\in \N$. Furthermore, $\supp \xi^q_0   \subseteq \overline{\G (M_q)}$ and $\supp \xi^q_n   \subseteq \overline{\G (M_{q+1})}$, for all $q\in \N$.
\end{coro}

Similarly, we can summarize as follows the statements corresponding to Theorem \ref{thm:H1K}, Corollary \ref{coro:H1K}, Corollary \ref{coro:H1Kq0} and Corollary \ref{coro:H1Kn2} applied to the Morse complex.
\begin{coro}
\label{coro:supp_x1_m}
Let $\{ X^u \}_{u\in \N^n}$ be an $n$-parameter exhaustive filtration of a cell complex $X$,
let $\V$ be a fixed discrete gradient vector field consistent with the filtration, and let $\{ M^u \}_{u\in \N^n}$ be the associated $n$-parameter filtration of the Morse complex $M$. Then the following facts hold.
\begin{enumerate}
    \item If $u\in \N^n$ satisfies $B_q(M^u) = \sum_{j=1}^n B_q(M^{u-e_j})$ and there exists a permutation $\rho \in \mathrm{Sym}(n)$ such that, for every $\ell \in [n]$,
\[
Z_q(M^{u-e_{\rho(\ell)}})\cap \left( \sum_{j< \ell}Z_q(M^{u-e_{\rho(j)}})\right) = \sum_{j< \ell}Z_q(M^{u-e_{\rho(j)}-e_{\rho(\ell)}}) ,
\]
then $\xi_1^q(u)=0$. As a consequence, the containment $\supp \xi_1^q \subseteq \G(M_{q+1})\cup \overline{\G(M_q)}$ holds.
\item In the case $q=0$, if $u\in \N^n$ is such that $B_0(M^u)=\sum_{j=1}^n B_0 (M^{u-e_j})$, then $\xi_1^0(u)=0$. As a consequence, the following containments hold: 
\[
\supp \xi_0^0 \subseteq \overline{\G(M_{0})}, \qquad
\supp \xi_1^0 \subseteq \G(M_{1}), \qquad
\bigcup_{i=1}^n \supp \xi_i^0 \subseteq \overline{\G(M_{1})}. 
\]
\item In the case $n=2$, if $u\in \N^2$ is such that $B_q(M^u)=\sum_{j=1}^n B_q (M^{u-e_j})$, then $\xi_1^q(u)=0$. The following containments hold: 
\[
\supp \xi_0^q \subseteq \overline{\G(M_{q})}, \qquad
\supp \xi_1^q \subseteq \G(M_{q+1}), \qquad
\supp \xi_2^q \subseteq \overline{\G(M_{q+1})}. 
\]
\end{enumerate}
\end{coro}

\section{Homological critical grades and support of Betti tables for bifiltrations}
\label{sec:2param}

In this section, we fix $n=2$ and study the support of the Betti tables of persistent homology modules associated with a one-critical bifiltration $\{ X^u \}_{u\in \N^2}$ of a cell complex $X$. 
In what follows, we make use of the notations introduced at the beginning of Section \ref{sec:support_xi}. Additionally, for $q\in \N$, let us set
\[
\C_q (X) \coloneqq \{ u\in \N^2 \mid \dim H_q (X^u, X^{u-e_1}\cup X^{u-e_2}) \ne 0 \} 
\]
and call it the set of \emph{$q$-homological critical grades} (see \cite{guidolin2023morse}). 
For any fixed $u\in\N^2$ and any $q\in \N$, let us recall the following known inequalities (see \cite[Corollary~1]{landi2022relative}, and \cite{guidolin2023morse} for a generalization to the case $n\ge 2$):
\begin{equation}
\label{eq:ineq_criticalgrade}
\xi_0^{q}(u) +\xi_1^{q-1}(u) -\xi_2^{q-1}(u) \le \dim H_q(X^u,X^{u-e_1}\cup X^{u-e_2}) \le  \xi_0^{q}(u) +\xi_1^{q-1}(u) +\xi_2^{q-2}(u) .
\end{equation}

To interpret the results of this section, we remark that $\C_q (X) \subseteq \G (X_q)$ and, more generally, if $M$ is the Morse complex associated with any discrete gradient vector field consistent with the filtration $\{X^{u}\}_{u\in \N^2}$, by \cite[Prop.~1]{landi2022relative} we have $\C_q (X) \subseteq \G (M_q)$. As we will show (Proposition \ref{prop:supp_xi_2param} and Corollary \ref{coro:supp_2par}), for bifiltrations we are able to bound the support of the Betti tables using the sets $\C_q (X)$ instead of the sets $\G (M_q)$, thus strengthening our general results of Section \ref{sec:support_xi} (cf.\ Corollary \ref{coro:supp_xi_m} and Corollary \ref{coro:supp_x1_m}).

First, we prove a technical result that crucially depends on the one-criticality assumption (Section \ref{subsec:multipers}) on the bifiltration.

\begin{lem}
\label{lem:onecrit}
Let $v\in \N^2$ and let $j\ne \ell$ in $\{ 1,2 \}$. Then, there is a short exact sequence of chain complexes
\[
0 \xrightarrow{} C_* (X^{v-e_{\ell}}, X^{v-e_1-e_2}) 
\xrightarrow{} C_* (X^{v}, X^{v-e_j})
\xrightarrow{} C_* (X^{v}, X^{v-e_1} \cup X^{v-e_2})
\xrightarrow{} 0 .
\]
\end{lem}
\begin{rmk}
\label{rmk:notN2}    
The statement has to be interpreted by setting $X^{v-e_1}=\emptyset$ if $v-e_1$ is not in $\N^2$, and similarly for $X^{v-e_2}$ and $X^{v-e_1 -e_2}$. We use this convention throughout this section. 
\end{rmk}
\begin{proof}
Without loss of generality, we prove the statement for $j=1$ and $\ell=2$.
The sequence
\[
0 \xrightarrow{} C_* (X^{v-e_1} \cup X^{v-e_2}, X^{v-e_1}) 
\xrightarrow{} C_* (X^{v}, X^{v-e_1})
\xrightarrow{} C_* (X^{v}, X^{v-e_1} \cup X^{v-e_2})
\xrightarrow{} 0
\]
associated with the triple $X^{v-e_1} \subseteq X^{v-e_1} \cup X^{v-e_2}\subseteq X^{v}$ is exact. 
Now we observe that, for any $q\in\N$, the relative chain modules of the pair $(X^{v-e_1} \cup X^{v-e_2}, X^{v-e_1})$ are
\begin{align*}
C_q (X^{v-e_1} \cup X^{v-e_2}, X^{v-e_1}) := \frac{C_q (X^{v-e_1} \cup X^{v-e_2})}{C_q (X^{v-e_1})} = \frac{C_q (X^{v-e_1}) + C_q(X^{v-e_2})}{C_q (X^{v-e_1})} \\
\cong \frac{C_q(X^{v-e_2})}{C_q (X^{v-e_1})\cap C_q(X^{v-e_2})} = \frac{C_q(X^{v-e_2})}{C_q(X^{v-e_1-e_2})} =: C_q (X^{v-e_2}, X^{v-e_1 -e_2}) ,
\end{align*}
where we used the classical isomorphism theorem for modules and, in the penultimate equality, the fact that  $C_q (X^{v-e_1})\cap C_q(X^{v-e_2}) = C_q(X^{v-e_1-e_2})$ as a consequence of the equality $X^{v-e_1}\cap X^{v-e_2} = X^{v-e_1-e_2}$ given by the one-criticality assumption on the filtration (see Remark \ref{rmk:one-criticality}). 
These isomorphisms between chain modules commute with the differentials of the chain complexes $C_* (X^{v-e_1} \cup X^{v-e_2}, X^{v-e_1})$ and $C_* (X^{v-e_2}, X^{v-e_1 -e_2})$, since they are induced by the differential of $C_*(X)$.
\end{proof}

\begin{coro}
\label{coro:n2inductivestep}
Let $v\in \N^2$, $q\in \N$ and $j\ne \ell$ in $\{ 1,2 \}$, and suppose that $H_q (X^{v}, X^{v-e_1} \cup X^{v-e_2})=0$. Then $H_q (X^{v}, X^{v-e_j})\ne 0$ implies $H_q (X^{v-e_{\ell}}, X^{v-e_1-e_2})\ne 0$.
\end{coro}
\begin{proof}
By Lemma \ref{lem:onecrit}, the following is a portion of a long exact sequence in homology: 
\[
 H_q (X^{v-e_{\ell}}, X^{v-e_1-e_2}) 
\xrightarrow{} H_q (X^{v}, X^{v-e_j})
\xrightarrow{} H_q (X^{v}, X^{v-e_1} \cup X^{v-e_2}) .
\]
Since $H_q (X^{v}, X^{v-e_1} \cup X^{v-e_2})=0$, the first map is surjective, and the claim follows immediately.
\end{proof}

To prove the final result of this section (Proposition \ref{prop:supp_xi_2param}), we first show directly that the support of the Betti table $\xi_2^{q-1}$ is contained in $\overline{\C_q (X)}$.

\begin{lem}
\label{lem:xi2_2param}
For all $q\in \N$, we have $\supp \xi_2^{q-1} \subseteq \overline{\C_q (X)}$. 
\end{lem}
\begin{proof}
Let $u\in \supp \xi_2^{q-1}$. We prove that there exists $\lambda \in \N$ such that 
\begin{equation}
\label{eq:relhom}
H_{q}(X^{u-\lambda e_1},X^{u-(\lambda +1) e_1}\cup X^{u-\lambda e_1 -e_2})\ne 0 .
\end{equation}
If condition (\ref{eq:relhom}) holds for $\lambda =0$, then $u\in \C_q (X)$.
Otherwise, since the same property can be proven with the roles of $e_1$ and $e_2$ interchanged, our claim follows by observing that $(u-\lambda e_1)\vee (u-\mu e_2)=u$, for every $\lambda, \mu \in \N$.

Assume that (\ref{eq:relhom}) is false (i.e.\ it is an equality) for all $\lambda \in \N$; 
then $H_q (X^{u-\lambda e_1},X^{u-\lambda e_1 -e_2})\ne 0$ implies $H_q (X^{u-(\lambda +1)e_1},X^{u-(\lambda +1)e_1 -e_2})\ne 0$ by Corollary \ref{coro:n2inductivestep} (applied with $v:= u-\lambda e_1$), and we can therefore use an inductive argument. The base case of the induction is $H_q (X^{u-e_1},X^{u-e_1 -e_2})\ne 0$ for $\lambda =1$, which holds because the hypothesis $u\in \supp \xi_2^{q-1}$ implies that $i_{q-1}^{u-e_1-e_2,u-e_1}:H_{q-1}(X^{u-e_1-e_2})\to H_{q-1}(X^{u-e_1})$ has nonzero kernel (see Section \ref{subsec:Kosz1}). Since $X^{u-\lambda e_1}=\emptyset =X^{u-\lambda e_1 -e_2}$ for a sufficiently large $\lambda$, we see that the induction leads to a contradiction.
\end{proof}

\begin{prop}
\label{prop:supp_xi_2param}
For all $q\in \N$, we have $\supp \xi_0^{q}\cup \supp \xi_1^{q-1} \cup \supp \xi_2^{q-1} \subseteq \overline{\C_q (X)}$.
\end{prop}
\begin{proof}
Let us assume that $u\notin \overline{\C_q (X)}$. In the first inequality of (\ref{eq:ineq_criticalgrade}), the term
$\dim H_q(X^u,X^{u-e_1}\cup X^{u-e_2})$ is zero by definition of $\C_q (X)$. By Lemma \ref{lem:xi2_2param}, $\xi_2^{q-1}(u)=0$, hence we have $\xi_0^{q}(u) +\xi_1^{q-1}(u)=0$, which is equivalent to $\xi_0^{q}(u) =\xi_1^{q-1}(u)=0$. 
\end{proof}

We observe that the inclusion $\supp \xi_0^q \subseteq \overline{\C_q (X)}$ can be proven directly, in a similar way to the proof of Lemma \ref{lem:xi2_2param}. Contrarily, a direct proof of the inclusion $\supp \xi_1^{q-1} \subseteq \overline{\C_q (X)}$ eludes us.

In conclusion, for bifiltrations, we can bound the support of Betti tables as follows.
\begin{coro}
\label{coro:supp_2par}
For all $q\in \N$, the Betti tables of degree $q$ satisfy
\[
\supp \xi_0^{q}\cup \supp \xi_1^{q} \cup \supp \xi_2^{q} \, \subseteq \, \overline{\C_q (X)} \cup \overline{\C_{q+1} (X)}.
\]
Furthermore, the union of the supports of all Betti tables satisfies 
\[
\bigcup_q \C_q (X) \, \subseteq \, \bigcup_{q,i} \supp \xi_i^{q} \, \subseteq \, \bigcup_q\overline{\C_q (X)} .
\]
\end{coro}
\begin{proof}
The first statement holds by Proposition \ref{prop:supp_xi_2param} and implies the second inclusion of the second statement. 
The first inclusion of the second statement follows from the second inequality of (\ref{eq:ineq_criticalgrade}), which implies that $\C_q (X) \, \subseteq \, \supp \xi_0^{q}\cup \supp \xi_1^{q-1} \cup \supp \xi_2^{q-2}$, for all $q\in \N$.
\end{proof}

We remark that the first statement of Corollary \ref{coro:supp_2par} is not a consequence of Theorem \ref{thm:supp_xi_m}, as for 2-parameter persistent homology modules it is known that $\C_q (X)$ can be strictly contained in $\G (M_q)$, for any choice of a discrete gradient vector field to determine the latter set of grades (see \cite[p.~2369]{landi2022relative} for an example). 

For $n>2$ parameters, we believe that exact sequences like those of Lemma \ref{lem:onecrit}, along with those induced in homology, can still be useful to study the relation between Betti tables and homological critical grades. In this case, however, these sequences assemble in much more complicated systems, and appropriately disentangling them would require a  different approach.

\section{Generalization to multi-critical filtrations}
\label{sect:not-one-crit}

In this last section, we discuss how the results of Section \ref{sec:support_xi} and Section \ref{sec:2param} can be generalized to an $n$-parameter filtration $\{X^u\}_{u\in \N^n}$ that is not one-critical (Section \ref{subsec:multipers}). Such filtrations are called \emph{multi-critical}.
As observed in Section \ref{sec:free}, one-criticality ensures that the chain complex associated with the filtration $\{X^u\}_{u\in \N^n}$ is composed of free $n$-graded $S$-modules. More specifically, for any $q\in \N$, $C_q\coloneqq \bigoplus_{u\in\N^n}C_q(X^u)$ is free and isomorphic to $\bigoplus_{\sigma \in X_q} S(-v_{\sigma})$, with $v_{\sigma}$ denoting the unique entrance grade of the cell $\sigma$. The persistent homology module $V_q=\bigoplus_{u\in\N^n}H_q(X^u)$ is then defined as the homology at the middle term of the sequence $C_{q+1}\xrightarrow{\partial_{q+1}}C_{q}\xrightarrow{\; \partial_{q}\;}C_{q-1}$ of free $n$-graded $S$-modules. For a multi-critical filtration $\{X^u\}_{u\in \N^n}$, the modules of this sequence are in general not free. Using results from  \cite{Chacholski2017}, one can however present $V_q$ as the homology at the middle term of a sequence of free $n$-graded $S$-modules $A\xrightarrow{f}B\xrightarrow{g}C$ satisfying $gf=0$, which enables applying our results. Below, we describe the strategy to construct such a sequence starting from a multi-critical filtration of a cell complex. For brevity, in this section we call a chain complex any sequence of (not necessarily free) $n$-graded $S$-modules $A\xrightarrow{f}B\xrightarrow{g}C$ with $gf=0$, observing that it can be viewed for example as the chain complex $\cdots \to 0 \to \ker f \to A\xrightarrow{f}B\xrightarrow{g}C \to \coker g \to 0 \to \cdots$.   

Let $\mathcal{X}=\{X^u\}_{u\in \N^n}$ be a multi-critical $n$-parameter filtration of a cell complex $X$. We suppose the filtration to be exhaustive, meaning that $X=\bigcup_{u\in\N^n} X^u$. For every fixed $q\in \N$, we denote by $\mathcal{X}_q=\{X^u_q\}_{u\in \N^n}$ the induced filtration of sets of $q$-cells. Following \cite[Sect.~4]{Chacholski2017}, we recall how to construct a free presentation of the $n$-graded $S$-module $C_q \coloneqq \bigoplus_{u\in \N^n}C_q(X^u)$.

For any cell $\sigma \in X_q$, the $n$-parameter filtration $\mathcal{X}_q[\sigma] = \{X^u_q[\sigma]\}_{u\in \N^n}$ of sets is defined by
\[
X^u_q[\sigma] = \begin{cases}
\{\sigma \} & \text{if } \sigma \in X^u_q, \\    
\;\; \emptyset & \text{if } \sigma \notin X^u_q.
\end{cases}
\]
Let $\operatorname{ent}(\sigma)\coloneqq \{u\in \N^n \mid \sigma \in X_q^u \smallsetminus \bigcup_{j=1}^n X_q^{u-e_j} \}$ denote the set of entrance grades\footnote{In \cite{Chacholski2017}, the notation $\operatorname{gen}(\sigma)$ is used for the set here denoted by $\operatorname{ent}(\sigma)$.} of $\sigma$. We recall that a filtration is one-critical if and only if $\operatorname{ent}(\sigma)$ has exactly one element, for every cell $\sigma$ of $X$.
The $\F$-linear span of the filtration $\mathcal{X}_q[\sigma]$ is the $n$-graded $S$-module $C_q[\sigma]=\bigoplus_{u\mid \sigma \in X_q^u} \F$, which is isomorphic to the monomial ideal $\langle x^v \mid v\in \operatorname{ent}(\sigma) \rangle$. As observed in \cite{Chacholski2017},
a free presentation of the $n$-graded $S$-module $C_q[\sigma]$ is given by
\[
\bigoplus_{v_0\ne v_1 \in \operatorname{ent}(\sigma)} S(-v_0 \vee v_1)\xrightarrow{\; \pi_0[\sigma]-\pi_1[\sigma] \;} \bigoplus_{v \in \operatorname{ent}(\sigma)} S(-v),
\]
where the $n$-graded homomorphism $\pi_i[\sigma]$ sends the generator $1_{v_0 \vee v_1}$ at grade $v_0 \vee v_1$ of $S(-v_0 \vee v_1)$ to $x^{v_0 \vee v_1 -v_i} 1_{v_i} \in S(-v_i)^{v_0 \vee v_1}$, for $i\in \{0,1\}$.

The $n$-graded $S$-module $C_q \coloneqq \bigoplus_{u\in \N^n}C_q(X^u)$, which is the $\F$-linear span of the filtration $\mathcal{X}_q$, is isomorphic to $\bigoplus_{\sigma \in X_q} C_q[\sigma]$. As already observed, if the filtration $\mathcal{X}_q$ is not one-critical, $C_q$ is not free. A free presentation of $C_q$ is given by:
\[
\bigoplus_{\sigma \in X_q} \left( \bigoplus_{v_0\ne v_1 \in \operatorname{ent}(\sigma)} S(-v_0 \vee v_1)\xrightarrow{\; \pi_0[\sigma]-\pi_1[\sigma] \;} \bigoplus_{v \in \operatorname{ent}(\sigma)} S(-v) \right) .
\]
In other words, $C_q$ is isomorphic to the cokernel of the $n$-graded homomorphism $\pi_0 - \pi_1 \coloneqq \bigoplus_{\sigma \in X_q} \left( \pi_0[\sigma]-\pi_1[\sigma] \right)$. To establish notations of modules and homomorphisms that will be used in what follows, we write this presentation of $C_q$ as
\begin{equation}\label{eq:freepresCq}
\begin{tikzcd}[column sep=large]
	{R_q} & {G_q} & {C_q},
	\arrow["{\pi_0-\pi_1}", from=1-1, to=1-2]
	\arrow["{p_q}", two heads, from=1-2, to=1-3]
\end{tikzcd}
\end{equation}
where $G_q \coloneqq \bigoplus_{\sigma\in X_q} \bigoplus_{v\in \operatorname{ent}(\sigma)}S(-v)$ and $R_q \coloneqq \bigoplus_{\sigma\in X_q} \bigoplus_{v_0 \ne v_1 \in \operatorname{ent}(\sigma)}S(-v_0\vee v_1)$.

Next, following \cite[Sect.~5]{Chacholski2017} we review how the $n$-parameter persistent homology module $V_q=\bigoplus_{u\in \N^n}H_q(X^u)$ associated with a multi-critical filtration $\mathcal{X}=\{X^u\}_{u\in \N^n}$ can be expressed as the homology of an explicitly constructed chain complex of free $n$-graded $S$-modules. 
Although the construction of \cite[Sect.~5]{Chacholski2017} is for $n$-parameter filtrations of simplicial complexes, it can readily be adapted to $n$-parameter filtrations of cell complexes, as we now explain.

Starting from the sequence of $n$-graded $S$-modules
\begin{equation}\label{eq:shortCq}
\begin{tikzcd}
	{C_{q+1}} && {C_{q}} && {C_{q-1}},
	\arrow["{\partial_{q+1}}", from=1-1, to=1-3]
	\arrow["{\partial_{q}}", from=1-3, to=1-5]
\end{tikzcd}
\end{equation}
consider $C_{q-1}\cong \bigoplus_{\sigma \in X_{q-1}}C_{q-1}[\sigma]$ and define the free $n$-graded $S$-module $D_{q-1}\coloneqq \bigoplus_{\sigma \in X_{q-1}} S$ and the $n$-graded homomorphism $\eta_{q-1}: C_{q-1}\to D_{q-1}$ given by the direct sum of the inclusions $C_{q-1}[\sigma]\hookrightarrow S$, for all $\sigma \in X_{q-1}$. Since $\eta_{q-1}$ is injective, replacing $\partial_q$ by the composition $\eta_{q-1}\partial_q$ in the sequence (\ref{eq:shortCq}) does not affect the homology $V_q$ at the middle term. Similarly, since the homomorphism $p_{q+1}:G_{q+1}\to C_{q+1}$ defined as in (\ref{eq:freepresCq}) is surjective, replacing $\partial_{q+1}$ by the composition $\partial_{q+1}p_{q+1}$ in the sequence (\ref{eq:shortCq}) does not affect the homology $V_q$. In other words, the homology at the middle term of 
\begin{equation}\label{eq:shortCq2}
\begin{tikzcd}
	{G_{q+1}} && {C_{q}} && {D_{q-1}},
	\arrow["{\partial_{q+1}p_{q+1}}", from=1-1, to=1-3]
	\arrow["{\eta_{q-1}\partial_{q}}", from=1-3, to=1-5]
\end{tikzcd}
\end{equation}
is isomorphic to $V_q$. 
Since $G_{q+1}$ is free (and hence projective) and $p_q$ is surjective, there exists an $n$-graded homomorphism $\delta_{q+1}:G_{q+1}\to G_q$ such that the triangle
\[
\begin{tikzcd}[column sep=large]
	& {G_{q}} \\
	{G_{q+1}} & {C_{q}}
	\arrow["{p_q}", from=1-2, to=2-2]
	\arrow["{\delta_{q+1}}", from=2-1, to=1-2]
	\arrow["{\partial_{q+1}p_{q+1}}"', from=2-1, to=2-2]
\end{tikzcd}
\]
commutes. 
The proof of \cite[Prop.~5.2]{Chacholski2017} carries over, showing that $V_q$ is isomorphic to the homology at the middle term of the following chain complex of free $n$-graded $S$-modules:
\begin{equation}\label{eq:shortABC}
\begin{tikzcd}[column sep=large]
	{R_{q}\oplus G_{q+1}} && {G_{q}} && {D_{q-1}}.
	\arrow["{[\pi_0-\pi_1 \;\; \delta_{q+1}]}", from=1-1, to=1-3]
	\arrow["{\eta_{q-1}\partial_{q}p_q}", from=1-3, to=1-5]
\end{tikzcd}
\end{equation}
We remark that the construction of this chain complex is not canonical, as it requires choosing a lift $\delta_{q+1}$.
Now we denote by $\G (G_{q})$ the set of grades of the generators of $G_{q}$, and by $\G (R_{q})$ the set of grades of the generators of $R_{q}$, for all $q$. Explicitly, they are the following subsets of $\N^n$:
\begin{equation}\label{eq:GqRq}
\begin{split}
\G (G_{q}) &= \{ v\in \N^n \mid v\in \operatorname{ent}(\sigma) \text{ for some } \sigma \in X_q \}, \\
\G (R_{q}) &= \{ w\in \N^n \mid w=v_0\vee v_1 \text{ with } v_0\ne v_1\in \operatorname{ent}(\sigma), \text{ for some } \sigma \in X_q \}.
\end{split}
\end{equation}
Our results of Section \ref{sec:support_xi} and Section \ref{sec:2param} can be applied to the persistent homology module $V_q$ of a multi-critical filtration $\mathcal{X}=\{X^u\}_{u\in\N^n}$ by replacing the chain complex (\ref{eq:shortCq}) of (not necessarily free) $n$-graded $S$-modules by the chain complex (\ref{eq:shortABC}) of free $n$-graded $S$-modules to present $V_q$ as the homology at the middle term. 
In particular, this affects the sets of entrance grades of cells: in degree $q$, the set $\G (G_{q})$ now plays the role of $\G (X_{q})$ in Section \ref{sec:support_xi}; similarly, $\G (R_{q}\oplus G_{q+1})= \G (R_{q})\cup \G(G_{q+1})$ now replaces the set $\G (X_{q+1})$. 
Lastly, we observe that, with the aim of reducing the involved chain complexes, one can replace the $n$-filtered cell complex $X$ with an $n$-filtered Morse complex $M$, consider (\ref{eq:shortCq}) to be the chain complex associated with $M$, and construct (\ref{eq:shortABC}) from it.

As an example of how the results on one-critical filtrations can be adapted, we state the generalization of Theorem \ref{thm:supp_xi_m} and Corollary \ref{coro:H1K} to the case of multi-critical filtrations.

\begin{prop}\label{prop:multi-crit-support}
Let $\{ X^u \}_{u\in \N^n}$ be a multi-critical $n$-parameter exhaustive filtration of a cell complex $X$. Then, for all $q\in \N$,
\[
\supp \xi^q_0   \subseteq \overline{\G (G_q)}, \qquad \supp \xi_1^q \subseteq \G(R_{q})\cup \G(G_{q+1})\cup \overline{\G(G_q)}, \qquad \supp \xi^q_n   \subseteq \overline{\G (R_{q}) \cup \G(G_{q+1})},
\]
and
\[
\bigcup_{i=0}^n \supp \xi^q_i   \subseteq \overline{\G(R_{q})\cup \G(G_{q+1})} \cup \overline{\G(G_q)} ,
\]
where the sets $\G(G_q)$, $\G(G_{q+1})$ and $\G(R_{q})$ are as in (\ref{eq:GqRq}). Furthermore, the same containments hold if the sets $\G(G_q)$, $\G(G_{q+1})$ and $\G(R_{q})$ are determined from $\{ M^u \}_{u\in \N^n}$ instead of $\{ X^u \}_{u\in \N^n}$,
where $\{ M^u \}_{u\in \N^n}$ is the $n$-parameter filtration of the Morse complex $M$ associated with any fixed discrete gradient vector field consistent with the filtration $\{ X^u \}_{u\in \N^n}$.
\end{prop}

\section*{Acknowledgments} 
This work was partially supported by the Wallenberg AI, Autonomous Systems and Software Program (WASP) funded by the Knut and Alice Wallenberg Foundation,  by the dBrain collaborative project at Digital Futures at KTH, by the Strategic Support grant of the Digitalisation Platform at KTH and by the Data Driven Life Science (DDLS) program funded by the Knut and Alice Wallenberg Foundation.
This work was partially carried out by the second author within the activities of ARCES (University of Bologna) and under the auspices of INdAM-GNSAGA.

\bibliographystyle{alpha}
\bibliography{BiblioDatabase}

\end{document}